\documentclass[10pt]{amsart}

\IfFileExists{srcltx.sty}{\usepackage{srcltx}}

\numberwithin{equation}{section}

\usepackage[latin1]{inputenc}
\usepackage{xspace,amssymb,amsfonts,euscript}
\usepackage{amsthm,amsmath}
\usepackage{palatino}
\usepackage{euscript}
\input xy \xyoption {all}

\RequirePackage{color}
\definecolor{myred}{rgb}{0.75,0,0}
\definecolor{mygreen}{rgb}{0,0.5,0}
\definecolor{myblue}{rgb}{0,0,0.65}

\RequirePackage{ifpdf}
\ifpdf
  \IfFileExists{pdfsync.sty}{\RequirePackage{pdfsync}}{}
  \RequirePackage[pdftex,
   colorlinks = true,
   urlcolor = myblue, 
   citecolor = mygreen, 
   linkcolor = myred, 
   pagebackref,
   bookmarksopen=true]{hyperref}
\else
  \RequirePackage[hypertex]{hyperref}
\fi

\RequirePackage{ae, aecompl, aeguill} 

    \def\AM{{\mathbb{A}}}
    
    \def\CM{{\mathbb{C}}}
    \def\DM{{\mathbb{D}}}
    
    \def\FM{{\mathbb{F}}}
  \def\gg{{\mathfrak g}}  
  \def\hg{{\mathfrak h}}  \def\HM{{\mathbb{H}}}
  \def\ig{{\mathfrak i}}  
    
    \def\KM{{\mathbb{K}}}
  \def\lg{{\mathfrak l}}  
    
    \def\NM{{\mathbb{N}}}
    \def\OM{{\mathbb{O}}}
  \def\pg{{\mathfrak p}}  \def\PM{{\mathbb{P}}}
    \def\QM{{\mathbb{Q}}}
    \def\RM{{\mathbb{R}}}
  \def\sg{{\mathfrak s}}  
    
  \def\ug{{\mathfrak u}}

    \def\ZM{{\mathbb{Z}}}

    \def\BC{{\mathcal{B}}}
    \def\CC{{\mathcal{C}}}
    
    \def\EC{{\mathcal{E}}}
    \def\FC{{\mathcal{F}}}
    \def\GC{{\mathcal{G}}}
    \def\HC{{\mathcal{H}}}
    \def\IC{{\mathcal{I}}}
    \def\JC{{\mathcal{J}}}
    \def\KC{{\mathcal{K}}}
    \def\LC{{\mathcal{L}}}
    
    \def\NC{{\mathcal{N}}}
    \def\OC{{\mathcal{O}}}
    \def\PC{{\mathcal{P}}}

    \def\TC{{\mathcal{T}}}

\def\XB{{\mathbf X}}    
\def\YB{{\mathbf Y}}

\def\a{\alpha}
\def\b{\beta}

\def\l{\lambda}
\def\L{\Lambda}

\newcommand{\nc}{\newcommand} \newcommand{\renc}{\renewcommand}

\newcommand{\rdots}{\mathinner{ \mkern1mu\raise1pt\hbox{.}
    \mkern2mu\raise4pt\hbox{.}
    \mkern2mu\raise7pt\vbox{\kern7pt\hbox{.}}\mkern1mu}}

\def\mini{{\mathrm{min}}}

\def\top{{\mathrm{top}}}

\def\re{{\mathrm{re}}}
\def\reg{{\mathrm{reg}}}

\def\red{{\mathrm{red}}}

\DeclareMathOperator{\Lie}{Lie}

\newcommand{\elem}[1]{\stackrel{#1}{\longto}}

\def\wt{\widetilde}
\def\ov{\overline}
\def\un{\underline}

\def\p{{}^p}

\def\to{\rightarrow}

\def\longto{\longrightarrow}

\def\onto{\twoheadrightarrow}
\nc{\triright}{\stackrel{[1]}{\to}}
\nc{\longtriright}{\stackrel{[1]}{\longto}}

\nc{\Br}{\mathcal{B}}
\nc{\HotRR}{{}_R\mathcal{K}_R}
\nc{\HotR}{\mathcal{K}_R}
\nc{\excise}[1]{}
\nc{\defect}{\text{df}}
\nc{\h}[1]{\underline{H}_{#1}}

\nc{\Ga}{\mathbb{G}_a} 
\nc{\Gm}{\mathbb{G}_m} 

\nc{\Perv}{{\mathbf{P}}}

\nc{\IH}{{\mathrm{IH}}}

\nc{\ic}{\mathbf{IC}}

\nc{\gl}{{\mathfrak{gl}}}
\renc{\sl}{{\mathfrak{sl}}}
\renc{\sp}{{\mathfrak{sp}}}

\nc{\HBM}{H^{BM}}

\DeclareMathOperator{\im}{{\mathrm{Im}}}
 \DeclareMathOperator{\Hom}{Hom}
\DeclareMathOperator{\supp}{supp} 
\DeclareMathOperator{\End}{End} 
 
\DeclareMathOperator{\Loc}{Loc_{f}}
\DeclareMathOperator{\id}{Id}
\DeclareMathOperator{\Bs}{BS}
\DeclareMathOperator{\rad}{rad}
\DeclareMathOperator{\rank}{rank}

\newtheorem{thm}{Theorem}[section]
\newtheorem{lem}[thm]{Lemma}

\newtheorem{prop}[thm]{Proposition}
\newtheorem{cor}[thm]{Corollary}

\theoremstyle{definition}
\newtheorem{defi}[thm]{Definition}
\newtheorem{ex}[thm]{Example}

\theoremstyle{remark}
\newtheorem{remark}[thm]{Remark}
\newtheorem{question}[thm]{Question}

\newcommand{\into}{\hookrightarrow}

\def\Iff{\Longleftrightarrow}

\def\uk{\underline{k}}

\def\pt{{\mathrm{pt}}}

\def\Gr{{\EuScript Gr}}

\def\Mod{\mathrm{Mod}}
\def\Vect{\mathrm{Vect}}

\DeclareMathOperator{\even}{even}
\DeclareMathOperator{\odd}{odd}

\begin{document}

\title{Parity sheaves}

\author{Daniel Juteau} \address{ LMNO, Universit\'e de Caen
  Basse-Normandie, CNRS, BP 5186, 14032 Caen, France}
\email{daniel.juteau@unicaen.fr}
\urladdr{http://www.math.unicaen.fr/~juteau}

\author{Carl Mautner} 
\address{Max-Planck-Institut f\"ur Mathematik, Vivatsgasse 7, 53111 Bonn, Germany}
\email{cmautner@mpim-bonn.mpg.de}
\urladdr{http://people.mpim-bonn.mpg.de/cmautner/}

\author{Geordie Williamson} \address{Max-Planck-Institut f\"ur Mathematik, Vivatsgasse 7, 53111 Bonn, Germany}
\email{geordie@mpim-bonn.mpg.de}
\urladdr{http://people.mpim-bonn.mpg.de/geordie/}

\thanks{D.~J. was supported by ANR Grant No.~ANR-09-JCJC-0102-01 and C. M. by an NSF postdoctoral fellowship.}

\maketitle

\section{Introduction}

\subsection{Overview}

In view of applications in geometric representation theory in positive
characteristic, we introduce parity sheaves, a class of constructible
complexes of sheaves on stratified varieties whose strata satisfy a
cohomological parity vanishing condition. We show the existence and
uniqueness of parity sheaves on several spaces arising in
representation theory, including generalised flag varieties, nilpotent
cones (at least for $GL_n$) and toric varieties.

With sheaf coefficients in a field of characteristic zero, parity sheaves correspond to classical objects in geometric representation theory. When the coefficients are of positive characteristic, parity sheaves are important new objects. We show that parity sheaves, unlike intersection cohomology complexes, satisfy a form of the Decomposition Theorem, and explain the role played by intersection forms in determining the decomposition of their direct images. On flag varieties parity sheaves allow us to retrieve in a uniform way the Beilinson-Bezrukavnikov-Mirkovi{\'c} tilting sheaves and the special sheaves of Soergel, used by Fiebig in his proof of Lusztig's conjecture.

\subsection{Outline}
In Section~\ref{sec-dafp} we define parity sheaves and develop some of their basic properties.  Our notation and assumptions appear in~\ref{subsec-naa}.  The definition of parity sheaves appears as Definition~\ref{def-parsheaf} and depends on the preceding uniqueness result (Theorem~\ref{indecs}).  Section~\ref{subsec-exist} begins to explore the question of existence and gives a criterion for existence.  In Section~\ref{subsec-even}, we introduce the notion of an even map (Definition~\ref{def-even}) and show that the push forward functor along proper, even maps preserves the class of parity complexes (Proposition~\ref{prop-evenpush}).  This is our key tool for producing examples and serves as a weak analogue of the Decomposition Theorem.  Section~\ref{subsec:modular} is concerned with the behaviour of parity sheaves under modular reduction.  Proposition~\ref{prop:almostall} shows that when an $\ic$-sheaf with $\QM$-coefficients is parity, the corresponding modular $\ic$-sheaf is parity for all but finitely many characteristics.  Sections~\ref{sec-torsion} and~\ref{sec:ind} review respectively the notions of torsion primes and ind-varieties.

Section~\ref{sec-decomp} extends an observation of de Cataldo and Migliorini \cite{dCM, dCM2} from their recent Hodge theoretic proof of the Decomposition Theorem.
In their work a crucial role is played by the case of semi-small resolutions, and 
certain intersection forms attached to the strata of the target. Indeed,
they show that for a semi-small morphism 
the direct image of the intersection cohomology sheaf splits as a direct sum of
intersection cohomology complexes if and only if these forms are
non-degenerate.

In Section~\ref{subsec-intproduct}, we recall the definition of these intersection forms.  In Section~\ref{subsec-DTSS}, we extend the observation of de Cataldo and Migliorini and prove (Theorem~\ref{thm-intformsgeneral}) that the non-degeneracy of the modular reduction of these intersection forms (which are defined over the integers) determine exactly when the decomposition theorem fails in positive characteristic for a semi-small resolution.  Section~\ref{subsec-decompparity} addresses the case of a proper and even (but not necessarily semi-small) morphism from a smooth source.  Theorem~\ref{thm-paritydecomp} shows that the multiplicities of parity sheaves which occur in the direct image are given in terms of the ranks of these forms.  These theorems allow one to reformulate questions in representation theory in terms of such intersection forms.

The remaining sections explore three classes of examples: Kac-Moody
flag varieties (\ref{subsec-KacMoody}), toric varieties
(\ref{subsec-toric}) and nilpotent cones (\ref{subsec-Nilp}).

\subsection{Related work}
 The usefulness of some form of parity vanishing in
  representation theory and intersection/equivariant cohomology has
  been noticed by many authors (e.g. \cite{KL2}, \cite{SpIH},
  \cite{CPS}, \cite{GKM} and
  \cite{BJ}). In the following we comment briefly on
  ideas that are particularly closely related to the current work:

\subsubsection{Soergel's category $\KC$} The idea of considering another class of
objects as ``replacements'' for intersection cohomology complexes when
using positive characteristic coefficients is due to Soergel in
\cite{Soe}. He considers the full additive subcategory $\KC$ of the derived
category of sheaves of $k$-vector spaces on the flag variety which
occur as direct summands of direct images of the constant sheaf on
Bott-Samelson resolutions. Furthermore, he shows 
(using arguments from representation theory) that if the
characteristic of $k$ is larger than the Coxeter number, then 
the indecomposable objects in $\KC$ are parametrised by the Schubert
cells. In fact, the indecomposable objects in Soergel's category
$\KC$ are parity sheaves, and our arguments provide a
geometric way of understanding and expanding his result.

\subsubsection{Tilting perverse sheaves}
Since their introduction by Ringel \cite{RingelTilting}, an important role in
representation theory is played by tilting objects in highest weight
categories. There are several important examples of categories of perverse
sheaves which are highest weight, and it is desirable to have a local
(i.e. in terms of stalks and costalks) characterization of the tilting
sheaves. In \cite{TiltExer} Beilinson, Bez\-ru\-kav\-nikov and Mirkovi{\'c}
give such a description for the Schubert constructible perverse
sheaves on the flag variety. It is immediate from their description that the
tilting perverse sheaves can be also be characterised as parity sheaves
(for the ``dimension pariversity'', see Section \ref{subsec-defi}).

Another important example of a highest weight category of perverse
sheaves is the Satake category of $G[[t]]$-constructible perverse
sheaves on the affine Grassmannian. In \cite{JMW3} the authors show that
the parity sheaves on the affine Grassmannian correspond to the
tilting sheaves, under certain explicit bounds on the characteristic
of the coefficients. Thus, in two important and quite different examples---the finite flag variety
(or more generally any stratified variety satisfying the conditions of
\cite{TiltExer} and our parity conditions) and the affine
Grassmannian---we see that the indecomposable tilting sheaves are parity sheaves. Thus one is led to suspect a 
relation between parity sheaves and tilting sheaves on any
variety satisfying our parity conditions for which the corresponding
category of perverse sheaves is highest weight. For example one may
show that, in the above situation, if the parity sheaves for the
dimension pariversity are perverse then they are tilting sheaves.
In \cite{AM}, Achar and the second author start to explore this phenomenon in the case of nilpotent cones.

\subsubsection{Combinatorial models for intersection cohomology}
There exist combinatorial algorithms (due to Bernstein and Lunts \cite{BL} and 
Barthel, Brasselet, Fieseler and Kaup \cite{BBFK}) for calculating
the rational equivariant intersection cohomology of a toric variety using
commutative algebra. Similarly, Braden and MacPherson \cite{BrM}
gave an analogous combinatorial algorithm for Schubert varieties. In both cases
the calculation of intersection cohomology with modular coefficients
is significantly more difficult, and no algorithm is known. In \cite{FW}, Fiebig and the third author
show that, when performed with coefficients of positive
characteristics, the Braden-MacPherson algorithm computes the stalks
of the parity sheaves. It is likely that an analogous
result is true for toric varieties.

\subsubsection{The $p$-canonical basis}

Parity sheaves on generalised flag varieties with coefficients in a
field of characteristic $p \geq 0$ may be used to define a
``$p$-canonical basis'' for the Hecke algebra which enjoys remarkable
positivity properties (when $p = 0$ one recovers the Kazhdan-Lusztig
basis). In low rank examples, there are sufficiently many constraints
to force the $p$-canonical and Kazhdan-Lusztig bases to coincide
for all $p$ and almost all elements of the Weyl group
\cite{WillLowRank}. This indicates where to look for non-trivial
torsion, which does indeed occur. Braden had previously found
$2$-torsion in some Schubert varieties of types $A_7$ and $D_4$ (see
the appendix of \cite{WillLowRank}). More recently, Polo has found
$n$-torsion in a Schubert variety of type $A_{4n-1}$. The examples
in $A_7$ led to the discovery of a relation between non trivial parity
sheaves and the reducibility of characteristic varieties \cite{VW}.

There is a parallel story using Lusztig complexes on moduli spaces of
quiver representations, where one recovers the $p$-canonical basis for
the negative part of the quantised enveloping algebra \cite{Groj}. The
relationship with parity sheaves for linear quivers has been explained
by Maksimau \cite{Maksimau}. These results may be used to rephrase the
James conjecture in terms of parity sheaves.

\subsubsection{Weights and parity sheaves} Replacing our complex
variety $X$ by a variety $X_o$ defined over a finite field $\FM_q$,
one can consider Deligne's theory of weights in the derived category
$D^b_c(X_o, \QM_{\ell})$ of $\QM_{\ell}$-sheaves (see \cite{BBD} for
details and notation). 

In all examples considered in this paper, one can proceed naively, and say
that $\FC_o \in D^b_c(X_o, \ZM_{\ell})$ (resp. $D^b_c(X_o, \FM_{\ell})$) is
pure of weight 0 if $\HC^i(\FC)$ and $\HC^i(\DM \FC)$ vanish for odd
$i$ and, for all $x \in X_o(\FM_{q^n})$ the Frobenius $F_{q^n}^*$ acts on
the stalks of $\HC^{2i}(\FC)$ and $\HC^{2i}(\DM \FC)$ as
multiplication by $q^{ni}$ (the image of $q^{ni}$ in
$\FM_{\ell}$ respectively). With this definition one can show that, in all examples 
considered in this paper, there exist analogues of parity sheaves which
are pure of weight 0. Note
that the modular analogue of Gabber's theorem is not true: if $\FC_o$ in 
$D^b_c(X_o, \ZM_{\ell})$ or $D^b_c(X_o, \FM_{\ell})$ is pure of 
weight 0, then $\FC$ is not necessarily semi-simple.

Nevertheless, such considerations have been used by Riche, Soergel and the
third author to deduce that the dg-algebra of extensions of the direct
sum of all parity sheaves on the flag variety is formal. From this
they deduce a modular form of Koszul duality \cite{RSW}.

\subsection{Acknowledgements}
We would like to thank Alan Stapledon for help on the section about
toric varieties. We would also like to thank David Ben-Zvi, Matthew
Dyer, Peter Fiebig, Sebastian Herpel, Joel Kamnitzer, Frank L{\"u}beck, David Nadler,
Rapha{\"e}l Rouquier, Olaf Schn{\"u}rer, Eric Sommers, Tonny Springer,
Catharina Stroppel, Ben Webster and Xinwen Zhu for useful discussions
and comments. Finally, we thank Wilberd van der Kallen for pointing
out some mistakes in a previous version and two anonymous referees for a thorough reading.

We also thank the Centre International de Rencontres Math\'ematiques where parts of this paper were written.
The first and third authors would like to thank the Mathematical
Sciences Research Institute, Berkeley and the Isaac Newton Institute,
Cambridge for providing excellent research environments in which to
pursue this project. 
The second author would also like to thank his advisor David
Ben-Zvi, the geometry group at UT Austin, and David Saltman for
travel support during his time as a graduate student.

\section{Definition and first properties}
\label{sec-dafp}

\subsection{Notation and assumptions}
\label{subsec-naa}

Let $\OM$ denote a complete discrete valuation ring of characteristic
zero (e.g., a finite extension of $\ZM_p$), $\KM$ its field of
fractions (e.g., a finite extension of $\QM_p$), and $\FM$ its residue
field (e.g., a finite field $\FM_q$).
Unless stated otherwise, $k$ denotes a complete local principal ideal domain, which may be for
example $\KM$, $\OM$ or $\FM$, and all sheaves and cohomology groups
are to be understood with coefficients in $k$.

In what follows all varieties will be considered over $\CM$ and
equipped with the classical topology.
Throughout, $X$ denotes either a variety or a
$G$-variety for some connected linear algebraic group $G$. In
Sections~\ref{sec-dafp} and~\ref{sec-decomp} we deal with these two
situations simultaneously, bracketing the features which only apply in
the equivariant situation. In the examples, we will specify the set-up
in which we work.

We fix an algebraic stratification (in the sense of 
\cite[Definition 3.2.23]{CG})
\[ X = \bigsqcup_{\lambda \in \Lambda} X_{\lambda} \]
of $X$ into smooth connected locally closed ($G$-stable)
subsets. For each $\lambda \in \Lambda$ we denote by $i_{\lambda} :
X_{\lambda} \to X$ the inclusion and by $d_{\lambda}$ the complex
dimension of $X_{\lambda}$.

We denote by $D(X)$, or $D(X;k)$ if we wish to emphasise the
coefficients, the bounded (equivariant) constructible derived category
of $k$-sheaves on $X$ with respect to the given stratification
(see \cite{BLu} for the definition and basic properties of the equivariant
derived category). The category $D(X)$ is triangulated with shift
functor $[1]$. We call objects of $D(X)$ complexes.
For all $\lambda \in \Lambda$, let $\uk_{\lambda}$ denote 
the (equivariant) constant sheaf on $X_{\lambda}$. Given $\FC$ and
$\GC$ in $D(X)$ we set $\Hom(\FC, \GC) := \Hom_{D(X)}(\FC, \GC)$ and
$\Hom^n(\FC, \GC) := \Hom(\FC, \GC[n])$. We can form the graded $k$-module
$\Hom^{\bullet}(\FC, \GC) := \oplus_{n\in \ZM} \Hom^n(\FC, \GC)$.

Recall that an additive category is Krull-Remak-Schmidt if every
object is isomorphic to a finite direct sum of objects, each of which
has local endomorphism ring. In a Krull-Remak-Schmidt category all
idempotents split and any object admits a unique decomposition into
indecomposable objects. Moreover, an object is indecomposable if and
only if its endomorphism ring is local. By our assumptions on $k$,
$D(X)$ is a Krull-Remak-Schmidt category (see \cite{LeChen}).

\begin{remark} 
The category $D(X)$ is Krull-Remak-Schmidt as soon as the ring of coefficients $k$
is Noetherian and complete local. The Krull-Remak-Schmidt property of $D(X)$ is fundamental to
all arguments below. Above we make the stronger assumption that $k$
is a complete local principal ideal domain
(equivalently a field or complete discrete valuation ring). We use
this stronger assumption in Sections \ref{subsec-exist} and
\ref{subsec:modular}. The results of Sections \ref{subsec-defi},
\ref{subsec-even} and \ref{sec-examples} remain valid for coefficients
in any Noetherian complete local ring $k$.
In Section \ref{sec-decomp} we assume that $k$ is a field.
\end{remark}

For each $\lambda$, denote by $\Loc(X_{\lambda}, k)$ or $\Loc(X_{\lambda})$ the category of
(equivariant) local systems of free finite rank $k$-modules
on $X_{\lambda}$. We make the following
assumptions on our variety $X$, which are in force throughout the
paper except in Section \ref{subsec-DTSS} and \ref{subsec:min}. For each $\lambda \in \Lambda$ and all
$\LC,\LC'\in \Loc(X_{\lambda})$
we assume:
\begin{gather}
\label{assump-parity}
\text{ $\Hom^n(\LC, \LC^{\prime}) = 0$ for $n$ odd }
\end{gather}
and
\begin{gather}
\label{assump-parity2}
\text{ $\Hom^n(\LC, \LC')$ is a free $k$-module for all $n$.}
\end{gather}

\begin{remark} ~ \label{rem-par}
\begin{enumerate}
\item When $k$ is a field, all finite dimensional $k$-modules are free, so the
second assumption can be ignored.
\item Given two local systems $\LC, \LC^{\prime} \in \Loc(X_{\lambda})$ we
  have isomorphisms:
\[ \Hom^{\bullet}(\LC, \LC') \cong \Hom^{\bullet}(\uk_{\lambda},
\LC^{\vee} \otimes \LC') \cong
\HM^{\bullet}(\LC^{\vee} \otimes \LC'). \]
Hence \eqref{assump-parity} and \eqref{assump-parity2} are equivalent
to requiring that $\HM^{\bullet}(\LC)$ is a free $k$-module and
vanishes in odd degree, for all $\LC \in \Loc(X_{\lambda})$.
\item The condition \eqref{assump-parity} implies that there are no
  extensions between objects of the category $\Loc(X_\l)$.  In
  particular, if $k$ is a field, then $\Loc(X_\l)$ is semi-simple.
\end{enumerate}
\end{remark}

Finally, for $\lambda \in \Lambda$ and $\LC \in \Loc(X_\lambda)$, we
denote by $\ic(\lambda,\LC)$, or simply $\ic(\lambda)$ if $\LC = \uk_\l$,
the intersection cohomology complex on $\ov X_\lambda$ with
coefficients in $\LC$, shifted by $d_\lambda$ so that it is perverse,
and extended by zero on $X \setminus \ov X_\lambda$.
We always use the middle perversity $p_{1/2}$, which is self-dual when $k$ is a
field.  When $k$ is a ring of integers, it is not stable by duality,
so there is a dual $\ic$ for the dual $t$-structure $p_{1/2}^+$
\cite[\S 3.3]{BBD}. In this paper we only need the standard $\ic$.

\subsection{Definition and uniqueness}
\label{subsec-defi}

In this section the notation and assumptions are as in Section \ref{subsec-naa}.

\begin{defi}
  A \emph{pariversity} is a function $\dagger : \L \to \ZM /
  2$.\footnote{We regard elements of $\ZM/2$ as cosets and denote by
    $\overline{\cdot}:\ZM \to \ZM/2$ the non-trivial homomorphism, that is,
    $\overline{0}=\{n\in \ZM \mid n \textrm{ is even}\}$  and
    $\overline{1}=\{n\in \ZM \mid n \textrm{ is odd}\}$.}
\end{defi}

We will mainly be interested in two special pariversities: the
constant function $\natural$ defined by $\natural(\l)=\overline{0}$ for all $\l$ and
the dimension function $\Diamond$ defined by $\Diamond(\l)=
\overline{d_\l}$. Notice that if the strata are all even dimensional then $\natural=\Diamond$.

\begin{defi} Fix a pariversity $\dagger$. 
In the following $? \in \{ *, ! \}$.
\begin{itemize}
\item A complex $\FC \in D(X)$ is {\bf $(\dagger,?)$-even } (resp. {\bf
    $(\dagger,?)$-odd}) if,
  for all $\lambda \in \Lambda$ and $n \in \ZM$, the cohomology sheaf
  $\HC^n(i_\lambda^? \FC)$ belongs to $\Loc(X_\lambda)$ and
  vanishes for $n \not\in \dagger(\l)$ (resp. $n \in \dagger(\l)$).
\item A complex $\FC$ is {\bf $(\dagger,?)$-parity}
  if it is either $(\dagger,?)$-even or $(\dagger,?)$-odd.
\item A complex $\FC$ is {\bf $\dagger$-even} (resp. {\bf $\dagger$-odd}) 
  if it is both $(\dagger,*)$- and $(\dagger,!)$-even (resp. odd).
\item A complex $\FC$ is {\bf $\dagger$-parity} if it splits as the
  direct sum of a $\dagger$-even complex and a $\dagger$-odd complex.
\end{itemize}
\end{defi}

\medskip

\begin{remark} ~
\label{rmk-parity}
\begin{enumerate}
\item A complex is $(\dagger,*)$-even if and only if for every $\l\in\Lambda$ the stalks on $X_\l$ are free and concentrated in degrees $\dagger(\l)$.
\item By \eqref{assump-parity} and a
  standard d\'evissage argument, $\FC$ is $(\dagger,?)$-even (resp. odd) if and
  only if the $i_{\lambda}^{?}\FC$ are isomorphic to \emph{direct
  sums} of objects in $\Loc(X_\lambda)$ shifted by elements of $\dagger(\l)$ (resp. $\dagger(\l)+1$).
\item A complex $\FC$ is $(\dagger,*)$-even (resp. odd) if and only if $\DM \FC$ is
  $(\dagger,!)$-even (resp. odd).
\item An indecomposable $\dagger$-parity complex is either $\dagger$-even or $\dagger$-odd.
\item A complex is $\dagger$-parity if and only if it is
  $\dagger'$-parity, where $\dagger'(\lambda) = \dagger(\lambda) + \ov 1$ for
  all $\lambda\in\Lambda$.
\item If the pariversity function $\dagger$ is clear from the context,
  we may drop it from the notation.
\item This definition is a geometric analogue of a notion
  introduced by Cline-Parshall-Scott in~\cite{CPS}. For example, the
  notion of $*$-even corresponds to their $\EC^L$.
\end{enumerate}
\end{remark}

For the rest of the section, we fix a pariversity function $\dagger$
and drop $\dagger$ from the notation.

Given a $*$-even $\FC \in D(X)$ write $X' := \supp \FC$ for the
support\footnote{
Contrary to the common usage, we call support of a sheaf (or
of a complex) the \emph{closure} of the set of points where its
stalks are non-zero.
} of $\FC$ and choose an open stratum $X_{\mu} \subset
X^{\prime}$. We denote by $i$ and $j$ the inclusions:
\begin{equation*}
X_{\mu} \stackrel{j}{\hookrightarrow} X \stackrel{i}{\hookleftarrow} X^{\prime} 
\setminus X_{\mu}
\end{equation*}
We have a distinguished triangle of $*$-even complexes
\begin{equation}
\label{tri1}
j_! j^! \FC \to \FC \to i_* i^* \FC \triright
\end{equation}
which is the extension by zero of the standard distinguished triangle
on $X^{\prime}$. (Note that $j^! \FC = j^* \FC$ because $j$ factors as
an open immersion into $X'$ followed by the 
inclusion of $X'$ into $X$.)
Dually, if $\GC \in D(X)$ is $!$-even and $i$, $j$ are
as above we have a distinguished triangle of $!$-even complexes 
\begin{equation}
\label{tri2}
i_! i^! \GC \to \GC \to j_* j^* \GC \triright.
\end{equation}

\begin{prop} \label{prop-parityhom}
If $\FC$ is $*$-parity and $\GC$ is $!$-parity, then 
we have a (non-canonical) isomorphism of graded $k$-modules
\[ \Hom^{\bullet}(\FC, \GC) \cong \bigoplus_{\lambda \in \Lambda} \Hom^{\bullet}(i_
\lambda^* \FC, i_{\lambda}^! \GC). \]
Moreover, both sides are free $k$-modules.
\end{prop}

\begin{proof}
We may assume that $\FC$ and $\GC$ are indecomposable and, 
by shifting if necessary, that $\FC$ is $*$-even and
that $\GC$ is $!$-even.
We proceed by induction on the number $N$ of $\l\in\L$ such that
$i_\l^* \FC \neq 0$. If $N = 1$, then
$\FC \cong i_{\mu!}i_\mu^*\FC$ for some $\mu \in \L$, and by adjunction
\[
\Hom^\bullet (\FC,\GC) \cong \Hom^\bullet (i_{\mu!}i_\mu^*\FC,\GC)
\cong \Hom^\bullet (i_\mu^*\FC,i_\mu^! \GC).
\]

As we assumed $\FC$ to be $*$-even and $\GC$ to be $!$-even,
the complexes $i_\mu^*\FC$ and $i_\mu^! \GC$ are direct sums of shifts
of elements of $\Loc(X_\mu)$ concentrated in degrees congruent to
$\dagger(\l)$.  By \eqref{assump-parity} and \eqref{assump-parity2} we
conclude that $\Hom^\bullet (i_\mu^*\FC,i_\mu^! \GC)$ is free and
concentrated in even degrees.

If $N > 1$, applying $\Hom(-, \GC)$ to \eqref{tri1} yields
a long exact sequence 
\[ 
\dots \leftarrow \Hom^n(j_! j^!\FC, \GC) \leftarrow \Hom^n(\FC, \GC)
\leftarrow \Hom^n(i_* i^* \FC, \GC) \leftarrow \dots \] Now both
$\Hom^n(i_*i^*\FC,\GC)$ and $\Hom^n(j_!j^!\FC, \GC)$ vanish for $n$
odd and are free for $n$ even, respectively by induction and by the
base case, hence $\Hom^n(\FC,\GC)$ also vanishes for $n$ odd, and it
is an extension of $\Hom^n(i_* i^* \FC, \GC)$ by $\Hom^n(j_!
j^!\FC,\GC)$, hence also free, for $n$ even.
\end{proof}

\begin{remark}
The proof above shows that $\Hom^\bullet(\FC,\GC)$ is a
free $k$-module. If, moreover, $\FC$ and $\GC$ are indecomposable, then
the stalks (resp. costalks) of $\FC$ (resp. $\GC$) are concentrated in
degrees congruent to a fixed parity $a$ (resp. $b$) in $\ZM/2$, and
it follows that $\Hom^\bullet(\FC,\GC)$ is concentrated in
degrees congruent to $a + b$ modulo $2$.
\end{remark}

\begin{cor}
\label{prop-vanish}
If $\FC$ is $*$-even and $\GC$ is $!$-odd then
\begin{equation*}
\Hom(\FC, \GC) = 0.
\end{equation*}
\end{cor}

\begin{cor}
\label{cor-surj}
If $\FC$ and $\GC$ are indecomposable parity complexes of the same
parity and $j:X_{\mu}\to X$ denotes the inclusion of a stratum which
is open in the support of both $\FC$ and $\GC$, then the functor $j^*$ gives a surjection:
\begin{equation*}
\Hom(\FC, \GC) \onto \Hom(j^* \FC, j^* \GC).
\end{equation*}
\end{cor}

\begin{proof}
Apply $\Hom(\FC, -)$ to \eqref{tri2} and use Corollary \ref{prop-vanish}.
\end{proof}

The last corollary says that we can extend morphisms $j^*\FC \to j^*
\GC$ to morphisms $\FC \to \GC$. Now we want to investigate how
parity complexes behave when restricted to an open union of
strata. Before stating the result, let us recall the following simple result
from ring theory (whose proof is left as an exercise):

\begin{lem} \label{lem-quotient local}
  A quotient of a local ring is local.
\end{lem}

\begin{prop}
\label{prop-restriction}
Let $j: U \to X$ denote the inclusion of an open union of strata.  Then given an
indecomposable parity complex $\PC$ on $X$, its restriction to $U$ is
either zero or indecomposable.
\end{prop}

\begin{proof} 
Suppose that $\PC$ has non-zero restriction to $U$.
As in the proof of Corollary \ref{cor-surj}, the functor $\Hom(\PC,-)$ applied to the appropriate adjunction triangle together with Corollary \ref{prop-vanish} shows that restriction yields a
surjection
\[ \End(\PC) \twoheadrightarrow \End(\PC_{|U}).\]
It follows by Lemma \ref{lem-quotient local} that $\End(\PC_{|U})$ is
a local ring, and hence $\PC_{|U}$ is indecomposable.
\end{proof}

\begin{thm}
 \label{indecs}
Let $\FC$ be an indecomposable parity complex. Then
\begin{enumerate}
\item \label{irr supp}
the support of $\FC$ is irreducible, hence of the form
$\ov X_\lambda$, for some $\lambda \in \Lambda$;

\item \label{indec loc}
the restriction $i_\lambda^* \FC$ is isomorphic to $\LC[m]$, for some
indecomposable object $\LC$ in $\Loc(X_\lambda)$ and some integer $m$;

\item \label{unique}
any indecomposable parity complex supported on $\ov X_\lambda$ and
extending $\LC[m]$ is isomorphic to $\FC$.
\end{enumerate}
\end{thm}

\begin{proof} 
Suppose for contradiction that $X_\lambda$ and $X_\mu$ are open in
the support of $\FC$, where $\lambda$ and $\mu$ are two
distinct elements of $\Lambda$. Let $U = X_\lambda \cup X_\mu$.
Then $\FC_{|U} \simeq \FC_{X_\lambda} \oplus \FC_{X_\mu}$, contradicting
Proposition \ref{prop-restriction}. This proves \eqref{irr supp}.
The assertion \eqref{indec loc} also follows from Proposition \ref{prop-restriction}.

Now let $\GC$ be an indecomposable parity complex supported on 
$\ov X_\lambda$ and such that $i_\lambda^*\GC \simeq \LC[m]$. By composition,
we have inverse isomorphisms
$\alpha : i_\lambda^*\FC \elem{\sim} i_\lambda^*\GC$
and $\beta : i_\lambda^*\GC \elem{\sim} i_\lambda^*\FC$.
By Corollary \ref{cor-surj}, the restriction
$\Hom(\FC,\GC) \to \Hom(i_\lambda^*\FC,i_\lambda^*\GC)$ is surjective.
So we can lift $\alpha$ and $\beta$ to morphisms
$\tilde\alpha : \FC \to \GC$ and $\tilde\beta : \GC \to \FC$.
By Corollary \ref{cor-surj} again, the restriction
$\End(\FC) \to \End(i_\lambda^*\FC)$ is surjective.
Since $\tilde\beta \circ \tilde\alpha$ restricts to
$\beta \circ \alpha = \id$, the locality of $\End(\FC)$ implies
that $\tilde \beta \circ \tilde\alpha$ is invertible itself, and similarly for
$\tilde \alpha \circ \tilde \beta$. This proves \eqref{unique}.
\end{proof}

\begin{remark}
If $k$ is a field, one can replace ``indecomposable'' by ``simple'' in
\eqref{indec loc}, due to our assumptions on $X$.
\end{remark}

We now introduce the main character of our paper.

\begin{defi}
\label{def-parsheaf}
Let $\dagger$ be a pariversity. A $\dagger$-\textbf{parity sheaf} is
an indecomposable $\dagger$-parity complex with $X_{\lambda}$ open in
its support and extending $\LC[d_\lambda]$ for some indecomposable
$\LC\in \Loc(X_\lambda)$.  When such a complex exists, we will denote
it by $\EC^\dagger(\lambda,\LC)$. We call $\EC^\dagger(\lambda,\LC)$
the $\dagger$-parity sheaf associated to the pair $(\lambda,\LC)$.
\end{defi}

\begin{remark}
\begin{enumerate}
\item More generally, for $\LC$ not indecomposable, we will let
$\EC^\dagger(\lambda,\LC)$ denote the direct sum of the parity sheaves
associated to the direct summands of $\LC$.
We may also use the notation $\EC^\dagger(\ov X_\lambda,\LC)$.
\item If $\LC = \uk_{X_\l}$ is the constant local system, we may write
  $\EC^\dagger(\lambda, k)$, (or even $\EC^\dagger(\lambda)$ if the
  coefficient ring $k$ is clear from the context).
\item If the pariversity is clear from the context, we may also drop it from our notation.
\end{enumerate}
\end{remark}

Thus, any indecomposable parity complex is isomorphic to some
shift of a parity sheaf $\EC(\lambda,\LC)$. The reason for the
normalisation chosen in the last definition is explained
by the following proposition:

\begin{prop}
For any pariversity $\dagger$, $\lambda$ in $\Lambda$ and $\LC$ in $\Loc(X_\lambda)$,
we have
\[
\DM \EC^\dagger(\lambda,\LC) \simeq \EC^\dagger(\lambda,\LC^\vee).
\]
\end{prop}

\begin{proof}
The definition of $\dagger$-parity sheaf is clearly self-dual, so
$\DM \EC^\dagger(\lambda,\LC)$ is a $\dagger$-parity sheaf. Moreover, it
is supported on $\ov X_\lambda$ and extends
$\LC^\vee[d_\lambda]$. By the uniqueness theorem,
it is isomorphic to $\EC^\dagger(\lambda,\LC^\vee)$.
\end{proof}

\begin{remark}
\label{rem-vague?}
We give many examples of parity sheaves below.  In the
next section we also give a few examples of situations in which a full
set of parity sheaves does not exist, and, at the end of the paper,
examples of parity sheaves that are not perverse (see Proposition
\ref{prop:min}).

Despite such examples, in many cases of interest, parity sheaves do
exist and are perverse.  For example, if $X$ is a flag variety
stratified by Schubert cells and $k$ a field of characteristic zero,
the $\natural$-parity sheaves are the intersection cohomology complexes,
while the $\Diamond$-parity sheaves are the indecomposable tilting sheaves.
\end{remark}

\subsection{Parity extensions and an existence criterion}
\label{subsec-exist}

We have now explained that associated to each 
indecomposable local system on a stratum, there exists (up to 
isomorphism) at most one parity extension for a fixed pariversity.  In 
this section, we begin to explore when such an extension does in fact 
exist.

Following the construction of the intersection cohomology sheaves in \cite{BBD}, we present a
criterion for the existence of parity sheaves and, when they
exist, an explicit construction. This section concludes with some
examples of spaces and pariversities for which parity extensions do not
exist.

We begin with a lemma that characterises extensions of a complex from
an open set. 
We will use it below to develop our existence criterion.

\subsubsection{Extensions from an open set}

Let $j:U \to X$ be an open embedding and $i:Z \to X$ be the closed 
embedding of the complement. Recall that an extension of a complex $\FC_U \in D(U)$ is a pair $(\FC, \alpha)$ where $\FC \in D(X)$ is a complex and $\alpha : j^*\FC \stackrel{\sim}{\to} \FC_U$ is an isomorphism. Extensions of $\FC_U$ form a category in a natural way.

\begin{lem}
  Fix $\FC_U \in D(U)$. There is a natural bijection between isomorphism classes of extensions $(\FC, \a)$ of $\FC_U$ and isomorphism classes of distinguished triangles in $D(Z)$ of the form:
\[
A \to i^*j_*\FC_U \to B \triright
\]
Under this bijection we have $i^*\FC \cong A$ and $i^!\FC \cong B[-1]$.
\end{lem}

\begin{proof} We describe the maps in both directions. We leave it to the reader to check that these maps do indeed provide a bijection on isomorphism classes.

Suppose first that we are given an extension $(\FC,\alpha)$. Then we associate to $(\FC,\alpha)$ the distinguished triangle
\begin{equation} \label{eq:triext1}
i^*\FC \to i^*j_* \FC_U \to i^!\FC[1] \triright
\end{equation}
obtained by rotating the standard distinguished triangle $i^!\FC \to i^*\FC \to i^*j_*j^*\FC \triright$ \cite[1.4.7.2]{BBD} and using our isomorphism $\a$. It is clear that the isomorphism class of the resulting triangle depends only on the isomorphism class of the extension $(\FC, \a)$.

In the other direction, given a distinguished triangle:
\begin{equation} \label{eq:triext2}
A \to i^*j_* \FC_U \to B\triright
\end{equation}
we can certainly build an octahedron:
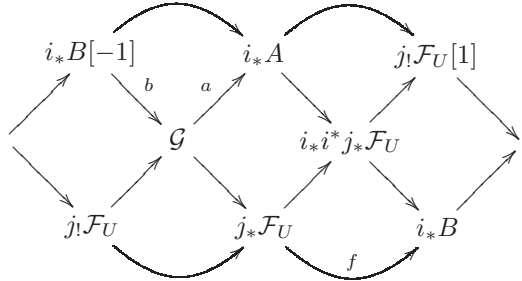
\vspace{-3cm}
\begin{equation*}
  \xymatrix@dr{
&& & \\
    && {j_!\FC_U[1]} 
\ar[r]^-{} & &   \\ 
     &{i_*A} \ar@(ur,ul)[ru]^-{}
    \ar[r]^-{} & 
    {i_*i^*j_*\FC_U} \ar[r]^-{} \ar[u]^-{} & 
    {i_*B} 
    \ar[u]_{} & \\ 
    i_*B[-1] \ar[r]^-{b} \ar@(ur,ul)[ru] & \GC \ar[u]^-{a} \ar[r]^-{} 
    & {j_*\FC_U} \ar@(dr,dl)[ru]^-{f} \ar[u]_{} \\  
    \ar[u] \ar[r] &{j_! \FC_U} \ar[u]^-{} \ar@(dr,dl)[ru]^-{} \\
    &
  }
\end{equation*}
To such a distinguished triangle we associate the extension $(\GC, \b)$, where $\b$ is obtained by adjunction from the map $\GC \to j_*\FC_U$.

The final statement of the lemma is clear by construction.
\end{proof}

\begin{remark}
Note that for $\FC_U$ perverse, the perverse extensions $\p j_! \FC_U$,
$\p j_{!*} \FC_U$ and $\p j_* \FC_U$ correspond to the perverse truncation
triangles
\[
\p\tau^{<n} i^*j_* \FC_U \to i^*j_* \FC_U \to \p\tau^{\geq n} i^*j_* \FC_U \triright,
\]
where $n = -1$, $0$, $1$ \cite[Proposition 1.4.23]{BBD}.
\end{remark}

\subsubsection{Parity extensions}

Having considered the case of general extensions, we now turn our
attention to the question of when parity extensions exist and how to
construct them. For the rest of this section we use the notation and
assumptions of Section \ref{subsec-naa}.

We proceed by imitating Deligne's construction of the intersection cohomology sheaves.
The method is a descending induction on the poset of strata.  The following corollary describes the inductive step.

\begin{cor}
\label{cor-parext}
Let $Z$ be a closed stratum. Fix a
pariversity $\dagger$ and let $\FC_U\in D(U)$ be a $\dagger$-even
complex.  There exists a $\dagger$-even extension of $\FC_U$ if and
only if there exists a distinguished triangle in $D(Z)$ of the form
\[
A \to i^*j_* \FC_U \to B \triright
\]
where $A$ is $\dagger$-even and $B$ is $\dagger$-odd.
\end{cor}

\begin{proof} Immediate from the previous lemma.
\end{proof}

To identify when such extensions are indecomposable, we need the following lemma.

\begin{lem} \label{lem:summandsupport}
A complex $\FC\in D(X)$ has no summands supported within a closed
subset $i:Z \to X$ if and only the map $i^!\FC \to i^*\FC$ cannot be
expressed as a direct sum $\id_Q \oplus\, h$, where $i^!\FC \cong Q
\oplus \GC$, $i^*\FC \cong Q \oplus \GC'$ and $h:\GC \to \GC'$, with $Q\neq 0$.
\end{lem}

\begin{proof}
Assume that $\FC \cong i_*Q \oplus \FC'$.  Then
$i^? \FC \cong Q \oplus i^? \FC'$ and the co\-res\-tric\-tion-to-res\-tric\-tion
map decomposes as a direct sum of the co\-res\-tric\-tion-to-res\-tric\-tion map
for $i_*Q$ (which is simply the identity map $\id_Q$) and that for
$\FC'$.

Conversely, suppose that the corestriction-to-restriction map for
$\FC$ can be expressed as $\id_Q \oplus h$ as above.  As the
corestriction-to-restriction map factors through $\FC$, we find that
there is a commutative diagram:
\[
\xymatrix{
& i_* i^! \FC \ar[r]^{\sim} \ar[d] \ar[ld]
& i_*(Q \oplus \GC) \ar[d]^{\id_Q \oplus h}
& i_* Q \ar[d]^{=} \ar[l]
\\
\FC \ar[r]
& i_* i^*\FC \ar[r]^{\sim}
& i_* (Q\oplus \GC') \ar[r]
& i_* Q
}
\]
It follows that $i_* Q$ is a direct summand of $\FC$.
\end{proof}

We now return to the question of parity extensions and the notation of
Corollary \ref{cor-parext}. As $A$ and $B$ are parity,
\eqref{assump-parity} and \eqref{assump-parity2} imply that $A\cong \oplus_{n} \HC^{n}(A)[-n]$
and $B \cong \oplus_{n} \HC^{n}(B)[-n]$.  Let $\phi_n$ be the
composition of the inclusion $\HC^{n}(B)[-n]\hookrightarrow B$, the
connecting map $B \to A[1]$, and the projection $A[1] \to
\HC^{n+1}(A)[-n]$.

\begin{cor}
\label{cor-parextind}
The parity extension $\FC$ is indecomposable if and only the image of each morphism $\phi_n$ defined above does not contain any non-zero direct-summand of $H^{n+1}(A)$.  If $k$ is a field, this condition is equivalent to $\phi_n=0$ for all $n$.
\end{cor}

\begin{proof}
This statement is an application of the previous lemma together with the fact that $H^n(A)$ is in $\Loc(Z)$ and therefore projective in the category of all local systems.  

In this case, as mentioned in the remark above, the corestriction and
restriction are both isomorphic to the direct sum of (the shifts of)
their cohomology sheaves.  Thus the condition of the lemma can be
checked in each degree separately.  Now for a single degree, the map $\phi_n$ maps to a projective object, so any direct summand contained in the image of the map is also a direct summand of the source.  The result follows.
\end{proof}

\begin{prop}
\label{claim:both exist}
Assume again that $Z$ consists of a single stratum and that $\FC_U$ is a $\dagger$-even complex on
$U$, where $\dagger$ is some pariversity on $X \setminus Z$.
Then there exists a $\tilde\dagger$-even complex $\FC$ on $X$
extending $\FC_U$ for both pariversities $\tilde\dagger$ extending $\dagger$
if and only if the complex $i^*j_*\FC_U\in D(Z)$ is parity (i.e. isomorphic
to a direct sum of shifts of elements of $\Loc(Z)$).
\end{prop}

\begin{proof}
If $i^*j_*\FC_U$ is isomorphic to a direct sum of shifts of objects in
$\Loc(Z)$, then it is clear that the required triangles exist and so
both parity extensions exist.

Suppose that there exist parity extensions of each parity.  Then there
are triangles
\[
A \to i^*j_*\FC_U \to B\triright\qquad \text{and}\qquad A' \to i^*j_*\FC_U \to B'\triright
\]
with $A,B'$ even and $A',B$ odd.  Without loss of generality, we
may assume that the parity extensions have no non-zero summands with support contained in
$Z$.  By Corollary \ref{cor-parextind}, this means that the images of the
morphisms $\phi_{2n-1}:\HC^{2n-1}(B) \to \HC^{2n}(A)$ (resp.
$\phi'_{2n}:\HC^{2n}(B') \to \HC^{2n+1}(A')$) do not contain a non-zero direct summand of $\HC^{2n}(A)$ (resp. $\HC^{2n+1}(A')$).

Consider the associated long exact sequences of cohomology sheaves for
the above triangles.  We find exact sequences
\[
0 \to \im(\phi_{2n-1}) \to \HC^{2n}(A) \to \HC^{2n}(i^*j_*\FC_U) \to 0,
\]
\[
0 \to \im(\phi'_{2n}) \to \HC^{2n+1}(A') \to \HC^{2n+1}(i^*j_*\FC_U) \to 0,
\]
as well as inclusions
\[
0 \to \HC^{2n}(i^*j_*\FC_U) \to \HC^{2n}(B'),\quad
0 \to \HC^{2n+1}(i^*j_*\FC_U) \to \HC^{2n+1}(B).
\]
As $\HC^{2n}(i^*j_*\FC_U)$ (resp. $\HC^{2n+1}(i^*j_*\FC_U)$) is a
subobject of $\HC^{2n}(B') \in \Loc(Z)$ (resp. $\HC^{2n+1}(B)$), it is
an element of $\Loc(Z)$ (remember that our ring of coefficients is a
PID by assumption).  In particular, it is a projective object in
the category of all local systems on $Z$, by Remark \ref{rem-par}(3).  It follows that the short
exact sequences above split and thus $\im(\phi_{2n-1})$ and
$\im(\phi_{2n})$ are direct summands. By our assumption above, they must be trivial.

We conclude that the induced maps $\HC^{2n}(A) \to
\HC^{2n}(i^*j_*\FC_U)$ (resp. $\HC^{2n+1}(A') \to
\HC^{2n+1}(i^*j_*\FC_U)$) are isomorphisms.  Thus the map $A\oplus A'
\to X$ induces isomorphisms on cohomology sheaves and therefore is an
isomorphism in $D(Z)$. The complex $A$ is even and $A'$ is odd, so
$i^*j_*\FC_U$ is a parity complex.
\end{proof}

\begin{remark}
\label{rem-directsum}
If the equivalent conditions of the above claim hold, and $\FC$ is a
parity extension with no summands supported on $Z$, then $i^*j_*\FC_U$
is isomorphic to the direct sum $i^* \FC \oplus i^! \FC[1]$.
\end{remark}

\begin{remark}
It is natural to ask whether the parity extension can be made into a
functor. Then the assignment of $A$ and $B$ to $\FC_U$ in the triangle
\begin{equation}
 \label{eq:extension triangle}
A \to i^*j_*\FC_U \to B \triright
\end{equation}
should be functorial as well. Though there are some situations where
this is the case, as we will see below (for example, when all strata
are contractible), in general there is no reason why it should be
true. By \cite[Proposition 1.1.9]{BBD} a sufficient condition
for the assignment of $A$ and $B$ to be functorial would be that $\Hom(A[1], B) = 0$ for any
two objects $A$ (resp. $B$) appearing on the left (resp. right) of
\eqref{eq:extension triangle} for any $\FC_U$. In the case of a parity extension, $A$ is
assumed to the $\dagger$-even, and $B$ $\dagger$-odd; hence
$\Hom(A[1], B) \ne 0$ in general, and this argument does not apply.

This is in contrast to the case of the $\p j_!$, $\p j_{!*}$ and $\p j_*$
extensions which are obtained via a truncation triangle, which is
functorial and satisfies $\Hom(A[1], B) = 0$. The three perverse
extensions can then be constructed by either $j_!$ or $j_*$ followed by a
functor of truncation on the closed part, which is obtained by gluing
a degenerate $t$-structure on the open part and a given $t$-structure
on the closed part \cite[Proposition 1.4.23]{BBD}.
\end{remark}

Lemma \ref{lem:summandsupport} gave a criterion for determining when a
complex has summands supported on a fixed closed subset.  The
following proposition refines the lemma in a particular
situation and will be useful when we come to discuss the
Decomposition Theorem in Section \ref{sec-decomp}.

\begin{prop} \label{prop:decompwithsupport}
Assume the ring of coefficients $k$ is a field.  Let $i:X_\l \into X$ be a
closed stratum.  Let $\FC\in D(X)$ be a complex such that $i^!\FC$ and
$i^*\FC$ are semi-simple.\footnote{We say a complex $Q\in D(F)$ is
  semi-simple if $Q\cong \bigoplus \HC^n(Q)[-n]$ and $\HC^n(Q)$ is
  semi-simple as an object of $\Loc(X_\l)$ for all $n$. (The
second condition is automatic, by Remark
\ref{rem-par}(3).)}
Let $D_\l$ denote the canonical
morphism
\[
D_\l  : i^! \FC \to i^* \FC
\]
and set
\[
D_\l^n := \HC^n(D_\l) : \HC^n(i^!\FC) \to \HC^n(i^*\FC).
\]
By assumption, both $\HC^n(i^!\FC)$ and $\HC^n(i^*\FC)$ are
semi-simple and so we can (and do) choose splittings
\begin{gather}
\label{eq:ss1}   \HC^n(i^!\FC)  \cong \ker D^n_\l \oplus (\HC^n(i^!\FC) / \ker
  D^n_\l),  \\
\label{eq:ss2} \HC^n(i^*\FC) \cong \im D^n_\l \oplus (\HC^n(i^*\FC) / \im D^n_\l ).
\end{gather}
Then there is an isomorphism
\[
\FC \cong  \GC \oplus \bigoplus_{n \in \ZM} (i_*( \HC^n(i^! \FC) / \ker
D^n_\l) [-n])
\]
where $\GC$ is a complex having no direct summand
supported on $X_\l$.
\end{prop}

\begin{proof} Fix isomorphisms $i^?\FC \cong \bigoplus
  \HC^n(i^?\FC)[-n]$. Composing these with $D_\l$ we obtain a map
\[
\bigoplus \HC^n(i^!\FC)[-n] \stackrel{\sim}{\to} i^!\FC \stackrel{D_\l}{\longrightarrow} i^*\FC
\stackrel{\sim}{\to} \bigoplus \HC^n(i^*\FC)[-n].
\]
The splittings \eqref{eq:ss1} and \eqref{eq:ss2} then yield a map
\begin{equation} \label{eq:sscomp}
c: \bigoplus (\HC^n(i^!\FC) / \ker D^n_\l)[-n] {\to} i^!\FC
\stackrel{D_\l}{\longrightarrow} i^*\FC {\to} \bigoplus (\im
D^n_\l)[-n].
\end{equation}
Note that $\HC^n(c)$ is an isomorphism for all $n$, and hence $c$ is
an isomorphism. Applying $i_! = i_*$ to \eqref{eq:sscomp} and using
that the natural map $i_!i^!\FC \to i_*i^*\FC$ factors through $\FC$
(as in the proof of Lemma \ref{lem:summandsupport}) we get that
$\bigoplus ((\HC^n(i^!\FC) / \ker D^n_\l)[-n])$ is a direct summand of
$\FC$. Hence we have an isomorphism
\begin{equation} \label{eq:ssdecomp}
\FC \cong \GC \oplus \bigoplus i_*((\HC^n(i^!\FC) / \ker D^n_\l)[-n])
\end{equation}
for some $\GC \in D(X)$.

It remains to see that $\GC$ does not have any direct summands
supported on $X_\l$. To this end, fix $n \in \ZM$ and let $\iota : \GC
\to \FC$ and $\pi : \FC \to i_*((\HC^n(i^!\FC) / \ker
D^n_\l)[-n])$ denote the inclusion and projections obtained from the
isomorphism \eqref{eq:ssdecomp}. If we
apply $\HC^n(i^! -)$ we see that $0 = \HC^n(i^! (\pi \circ \iota)) :
\HC^n(i^! \GC)
\to \HC^n(i^!\FC) / \ker D^n_\l$. Hence the image of $\HC^n(i^!
(\iota)) : \HC^n(i^! \GC) \to \HC^n(i^* \GC)$ is contained in $\ker
D^n_\l$. It follows that the map $\HC^n(i^! \GC) \to \HC^n(i^* \GC)$ obtained by
applying $\HC^n(-)$ to the natural map $i^! \GC \to i^* \GC$ is zero.

We conclude that $\HC^n(i^! \GC) \to \HC^n(i^* \GC)$ is zero for all
$n$ and hence that there is no direct summand of $i^!\GC$ 
mapped isomorphically onto a direct summand of $i^* \GC$. Hence $\GC$
does not have any direct summands supported on $X_\l$ by Lemma
\ref{lem:summandsupport}. This completes the proof of the proposition.
\end{proof}

\subsubsection{Constructible complexes with coefficients in a field
  for a stratification with contractible strata} \label{subsec:fieldcontract}

In this section we assume that we are in the non-equivariant setting and that $k$, our ring of coefficients, is a field. We use the construction from the previous
paragraph to prove that if all strata are contractible then parity sheaves exist
on all strata and for all pariversities.

\begin{lem} 
\label{prop:contractible existence}
Suppose that $Z$ is a contractible closed stratum of $X$ and let $U := X \setminus Z$ denote the complement. Fix a pariversity $\dagger$ on $U$ and let $\FC_U \in D(U)$ denote a $\dagger$-parity complex. Then there exists an extension $\FC$ of $\FC_U$ to $X$ whose restriction and corestriction to $Z$
  is even.  Similarly there exists an extension whose restriction
  and corestriction is odd.
\end{lem}

\begin{proof}  By Corollary \ref{cor-parext} it suffices to show that there exists a triangle
\[
A \to i^*j_* \FC_U \to B \triright
\]
with $A,B\in D(Z)$ respectively even and odd.  As $Z$ is contractible and $k$ a field, any object in $D(Z)$ is isomorphic to a direct sum of shifts of constant sheaves.  Thus $i^*j_* \FC_U \cong \HC^{\even}(i^*j_* \FC_U) \oplus \HC^{\odd}(i^*j_* \FC_U)$ and we have the split triangle:
\[
\HC^{\even}(i^*j_* \FC_U) \to i^*j_* \FC_U \to \HC^{\odd}(i^*j_* \FC_U) \triright.
\]
Under the bijection described in Corollary \ref{cor-parext} this triangle corresponds to the desired extension.  Entirely analogously, as the sequence splits, the maps can be directed in the other direction, producing an `odd' extension.
\end{proof}

\begin{cor}
\label{cor:contractible existence}
Let $\dagger$ be a pariversity.  Let $X$ be stratified by contractible strata and $k$ be a field.  Then for every stratum, there exists a parity sheaf $\EC^\dagger(\l,k)\in D(X)$.
\end{cor}

\subsubsection{Non-examples}
\label{subsec-nonexamples}

By modifying the conditions required in the previous paragraph, one
can find examples of spaces satisfying the parity conditions, for
which some parity extensions do not exist. 
We conclude this section by mentioning three such examples.

\begin{ex} 
Let $k=\ZM_2$, the ring of 2-adic integers. Consider the affine Grassmannian for
  $GL_2$, whose $GL_2(\CM[[t]])$-orbits $\Gr^\l$ are labelled by
  highest weights $(l_1,l_2)\in \ZM^2$ for $GL_2$.  Let $X$ be the
  orbit closure of $\Gr^\l$ for $\l=(1,-1)$.  The strata will be
  $U = \Gr^\l$ and $Z = \Gr^0 \cong \pt$.  The lower stratum $Z$ is contractible and
  the larger stratum $U$ has the cohomology of $\PM^1$, and thus satisfies
  the parity conditions (note that in Proposition
  \ref{prop:contractible existence}, only $Z$ is assumed to be
  contractible, not $U$).  Consider $i^*j_* \uk_U$.  It is
  quasi-isomorphic to the graded cohomology of the link, which is in
  turn homotopy equivalent to $\RM\PM^3$.  Now
  $H^*(\RM\PM^3;\ZM_2)=\ZM_2[-3]\oplus \ZM/2[-2] \oplus \ZM_2$.

While there is a distinguished triangle
\[
\ZM_2[-2]\oplus \ZM_2 \to \ZM_2[-3]\oplus \ZM/2[-2] \oplus \ZM_2 \to \ZM_2[-3] \oplus \ZM_2[-1]\triright,
\]
where the first term is even and the last term is odd (in particular,
both are free over $\ZM_2$), by Proposition \ref{claim:both exist}
there is no such triangle in the other direction (as $H^*(\RM\PM^3;\ZM_2)$ is not a
direct sum of free modules).  We conclude that there exists a parity
sheaf for the $\natural=\Diamond$ pariversity, but for not the pariversity
$\dagger(\l)=\overline{0}$, $\dagger(0)=\overline{1}$.
\end{ex}

\begin{ex}
Let $k$ be a field and $X$ the Hirzebruch surface $\PM(\OC_{\PM^1}
\oplus \OC_{\PM^1}(-2))$ with two strata: the zero section $E$ and its
complement $U$.  Let $\dagger$ be the pariversity that is even on $U$
and odd on $E$.  Similarly to the previous example, one finds that
$\uk_U$ has $\dagger$-parity extension if and only if the
characteristic of $k$ is equal to $2$.
\end{ex}

\begin{ex}
Lastly, let $k$ be a field and $X=\PM^1$ viewed  as a
  $\CM^\times$-variety with the standard action ($z \cdot [x:y] =
  [x:zy]$ for $z \in \CM^\times$). We fix the stratification indexed
  by $\Lambda = \{ 0, \infty \}$, where $X_0=\{[1:0]\}$ is the north pole and $X_\infty$ its complement.

Let $D(X) = D^b_{{\CM^\times}, \Lambda}(\PM^1, k)$ denote the bounded
$\Lambda$-constructible equivariant derived category.  The equivariant constant sheaf on $X_\infty$ does not have a $\Diamond$-parity extension in the equivariant derived category $D(X)$.  Note that in the non-equivariant derived category, such a $\Diamond$-parity extension does exist.
\end{ex}

\excise{

\begin{ex} Along similar lines, we can let $k$ be a field, but instead
  consider strata with higher cohomology.  Let $X$ be as in the
  previous example and let $\pi : \widetilde{X} \to X$ denote the blow
  up of $X$ at $\Gr^0$. Then $\pi$ is a resolution of singularities,
  and $\widetilde{X}$ is isomorphic to the Hirzebruch surface
  $\PM(\OC_{\PM^1} \oplus \OC_{\PM^1}(-2))$. Let us stratify $\widetilde{X} = U \sqcup E$ where $U =
  \pi^{-1}(\Gr^{(1,1)})$ and $E = \pi^{-1}(\Gr^0) \cong \PM^1$, and let $i$ and
  $j$ denote the inclusions
\[
E \stackrel{i}{\hookrightarrow} \widetilde{X} \stackrel{j}{\hookleftarrow} U.
\]
We already know that $\EC^\natural(\widetilde{X}) =
\uk_{\widetilde{X}}[2]$ exists. Hence, by Proposition \ref{claim:both exist}, if we consider a pariversity
$\dagger$ given by $\dagger(U) = \overline{0}$ and $\dagger(E) = \overline{1}$ then 
$\EC^\dagger(\widetilde{X})$ exists if and only if the object $i^*j_*
\uk_{\widetilde{X}}$ is isomorphic to the sum of its cohomology sheaves.

Now $E$ has an open neighbourhood $V$ in $X$ such that the inclusion
$E \hookrightarrow V$ is homeomorphic to the inclusion $\PM^1
\hookrightarrow \OC _{\PM^1} (-2)$. It follows that we have isomorphisms
\[
\uk_E \cong \tau_{\le 0} i^*j_* \uk_{\widetilde{X}}  \to 
i^*j_* \uk_{\widetilde{X}}  \to
\uk_E[-1] \cong \tau_{\ge 1} i^*j_* \uk_{\widetilde{X}}  \triright
\]
and the extension $\uk_E[-1] \to \uk_E[1]$ is identified with the
Chern class of $\OC _{\PM^1} (-2)$ in $H^2(\PM^1)$. It follows that
$i^*j_* \uk_{\widetilde{X}}$ is isomorphic to the direct sum of its
cohomology sheaves if and only if $k$ is of characteristic 2. It
follows that it is only in characteristic 2 that $\uk_{U}$ has a
$\dagger$-parity extension. (The reader is invited to take the direct
image under $\pi$ of the $\dagger$-parity extension in characteristic 2 and
investigate the relation with the previous example!)
\end{ex}

\begin{ex} In this example we take $X = \PM^1$ which we view as a
  $\CM^\times$-variety with the standard action ($z \cdot [x:y] =
  [x:zy]$ for $z \in \CM^\times$). We fix the stratification indexed
  by $\Lambda = \{ 0, \infty \}$:
\[
X = \{ 0 \} \sqcup \AM^1 = X_0 \sqcup X_{\infty} = \{ [1:0] \} \sqcup \{
[\gamma: 1] \; | \; \gamma \in \CM \}.
\]
Let $D(X) = D^b_{{\CM^\times}, \Lambda}(\PM^1, k)$ denote the bounded
$\Lambda$-constructible equivariant derived category and $i$ (resp. $j$)
the inclusion of $X_0$ (resp. $X_\infty$) into $X$.

Because $X$ is smooth, $\uk[1]$ is self-dual and
hence $\natural$-parity. It follows that we have an isomorphism $\EC^\natural(\CM) \cong \uk[1]$.

However, we claim that there is no $\Diamond$-parity sheaf supported on
all of $X$. By Corollary \ref{cor-parext} we need to examine the
sheaf $i^*j_* \uk_{X_\infty} \in D^b_{\CM^\times}(\pt)$. Note
that
\[
i^*j_* \uk_{X_\infty}   \cong i^{*}j^{'}_*{\uk_{\CM^\times}},
\]
where $j'$ denotes the inclusions of $\CM^\times = X_\infty \setminus
\{ \infty \}$. If we take $E\CM^\times = \CM^\infty \setminus
\{ 0 \}$ and apply the Borel construction to 
$\{ 0 \} \stackrel{i}{\hookrightarrow} \CM \stackrel{j'}{\hookleftarrow} \CM^\times$
we obtain the diagram
\[
\PM^\infty \stackrel{\tilde{i}}{\hookrightarrow}
\OC_{\PM^{\infty}}(-1) \stackrel{\tilde{j}}{\hookleftarrow} \CM^\infty
\setminus \{ 0 \}.
\]
Hence we have an isomorphism
\[
i^* j_* \uk_{X_\infty} \cong i^{*}j^{'}_*{\uk_{\CM^\times}}\cong \tilde{i}^* \tilde{j}_* \uk_{\CM^{\infty}
  \setminus \{ 0 \} }.
\]
(Recall that $D^b_{\CM^\times}(\pt)$ is by definition the full
subcategory of $D^b(\PM^\infty)$ consisting of complexes with locally
constant (hence constant) cohomology sheaves.)
Hence the cohomology sheaves of $i^* j_* \uk_{X_\infty}$ are of rank 1
in degrees 0 and 1, and zero otherwise. Now in the truncation
triangle
\[
\uk_{\PM^{\infty}} \to i^* j_* \uk_{X_\infty} \to
\uk_{\PM^{\infty}}[-1] \triright
\]
the extension $\uk_{\PM^{\infty}}[-1] \to \uk_{\PM^{\infty}}[1]$ is
the Chern class of the universal bundle $\CM^{\infty} \setminus \{ 0
\} \to \PM^{\infty}$ and in particular is $\ne 0$.\footnote{We remark
  that the existence of this truncation triangle 
gives another construction of $\EC^\natural(\CM)$, by Corollary \ref{cor-parext} .}
It follows that
$i^* j_* \uk_{X_\infty}$ is not isomorphic to the sum of its
cohomology sheaves, and hence no $\Diamond$-parity sheaf exists
extending $\uk_{X_\infty}$.

Let us remark that if we work instead with the (non-equivariant) derived
category $D^b_{\Lambda}(X,k)$ then we have already seen in
Corollary \ref{cor:contractible existence}  that the sheaf
$\uk_{X_\infty}$ does have a $\Diamond$-extension, which is an example
of a tilting perverse sheaf. The fact that this tilting sheaf does not
admit an equivariant lift reflects non-trivial facts in
representation theory. It also reminds us that the
image of the forgetful functor from the equivariant derived category
is not closed under extensions.
\end{ex}

}

\subsection{Even resolutions and existence} \label{subsec-even}

In this section we keep the notation and assumptions of Section \ref{subsec-naa}.

In Theorem \ref{indecs}, we saw that, if we fix a stratum $X_\lambda$ and a local system
$\LC$ on it, then there is at most one parity sheaf $\FC$ such that
$\supp\FC =\ov X_\lambda$ and $i_\lambda^*\FC \simeq \LC[d_\lambda]$,
up to isomorphism.  Furthermore, in Corollary \ref{cor:contractible
  existence} we showed that, for an arbitrary pariversity $\dagger$,
(in the non-equivariant set-up) if the strata are contractible and $k$
is a field, then $\dagger$-parity sheaves exist extending any local
system.  Unfortunately, the conditions of Corollary
\ref{cor:contractible existence} are quite restrictive. 

In this section, we will describe a very useful tool for constructing
$\natural$-parity sheaves that are extensions of the trivial local
system.  Later, in Section \ref{sec-examples}, we will use this tool
to provide large classes of examples.

Let us recall the following definition from \cite[1.6]{GM-SMT}.

\begin{defi} \label{def-stratified}
Let $X = \sqcup_{\lambda\in\Lambda_X} X_\lambda$
and $Y = \sqcup_{\mu\in\Lambda_Y} Y_\mu$ be stratified varieties.
A morphism $\pi : X \to Y$ is {\bf stratified} if
\begin{enumerate}
 \item 
for all $\mu\in\Lambda_Y$, the inverse image
$\pi^{-1}(Y_\mu)$ is a union of strata;

\item
for each $X_\lambda$ above $Y_\mu$, the induced morphism
$\pi_{\lambda,\mu} : X_\lambda \to Y_\mu$ is a submersion with smooth fibre
$F_{\lambda,\mu} = \pi_{\lambda,\mu}^{-1}(y_\mu)$,
where $y_\mu$ is some chosen base point in $Y_\mu$.
\end{enumerate}
\end{defi}

\begin{defi} \label{def-even}
A stratified morphism $\pi$ is said to be {\bf even} if for all
$\lambda$, $\mu$ as above, 
and for any local system $\LC$ in $\Loc(X_\lambda)$, the cohomology of the
fibre $F_{\lambda,\mu}$ with coefficients in $\LC_{|F_{\lambda,\mu}}$ is torsion
free and concentrated in even degrees.
\end{defi}

A class of even morphisms, which are common in geometric
representation theory and a motivation for the definition, are those
whose stratifications induce ``affine pavings'' on the fibres --- meaning
that all of the $F_{\lambda,\mu}$ are affine spaces. Examples of such
maps arise in the study of flag manifolds and will be discussed in
Section~\ref{sec-examples}.

\begin{prop}
\label{prop-evenpush}
The direct image of a $(\natural,?)$-even (resp. odd) complex 
under a proper, even morphism is again
$(\natural,?)$-even (resp. odd). The direct image 
of a $\natural$-parity complex under such a map
is $\natural$-parity.
\end{prop}

\begin{proof}
In this proof, all parity vanishing is with respect to the $\natural$
-pariversity and so we will drop it from the notation. We use the
notation of Definitions \ref{def-stratified} and \ref{def-even}, and
assume that our stratified even morphism $\pi$ is proper.

Note that if the statement is true for all $!$-even complexes,
then it is true for all $!$-odd complexes (by shifting). It would then
also be true for any $*$-parity complex because $\FC$ is $*$-even
(resp. odd) if and only if $\DM\FC$ is $!$-even (resp. odd) by
Remark~\ref{rmk-parity}
and  
\[\pi_* \DM\FC \cong \DM \pi_! \FC = \DM \pi_* \FC\]
as $\pi$ is proper.

The second sentence of the theorem is an immediate corollary of the first.
Thus it suffices to show the first statement for $!$-even sheaves.

Fix $\mu \in \Lambda_Y$ and let $\Lambda_X(\mu)$ denote the indices of strata in $\pi^{-1}(Y_{\mu})$. 
Thus we have $\pi^{-1}(Y_\mu) = \bigsqcup_{\lambda \in \Lambda_X(\mu)} X_{\lambda}$.
Moreover, let $F_\mu = \pi^{-1}(y_{\mu}) = \bigsqcup_{\lambda \in \Lambda_X(\mu)} F_{\lambda,
  \mu}$ where, as above, $F_{\lambda, \mu} = F_\mu \cap X_\lambda$.
We have the following diagram with Cartesian squares:
\[
\xymatrix{
F_\mu \ar[r] \ar[d]^{\pi} \ar@/^.6cm/[rr]^i
& \pi^{-1}(Y_\mu) \ar[r] \ar[d]^{\pi}
& X \ar[d]^{\pi}
\\
\{ y_{\mu} \} \ar[r] \ar@/_.6cm/[rr]_i
& Y_{\mu} \ar[r]
& Y
}
\]
We abuse notation and denote by $i$ both inclusions
$F_\mu \hookrightarrow X$ and $\{ y_{\mu} \}  \hookrightarrow
Y$. Similarly, $\pi$ denotes any vertical arrow in the above diagram.

Let $\PC \in D(X)$ be a $!$-even complex. We wish to show that
$\pi_* \PC$ is $!$-parity. This is equivalent to $i^! \pi_* \PC$ being
$!$-parity for all $\mu$.
By the proper base change theorem,
\[
i^! \pi_* \PC \cong \pi_* i^! \PC \cong \HM^{\bullet}(F_\mu, i^! \PC). \]
We will use the local-global spectral sequence to show that this
latter cohomology group is parity. 

Choose a filtration $F_0 \supset F_1 \supset F_2 \supset \dots \supset
F_r = \emptyset$ of the 
fibre $\pi^{-1}(y_{\mu})$ by closed subsets such that, for all $p$,
\[
F_p \setminus F_{p+1} = F_{\lambda_p, \mu} \qquad \text{for some } \lambda_p \in \Lambda_X(\mu).
\]
For all $p$, let $i_p : F_p \setminus F_{p+1} = F_{\lambda_p, \mu}
\hookrightarrow X$ denote the inclusion.
The local-global spectral sequence (see for example the proof of Proposition 3.4.4 in
\cite{Soergel-Langlands}) has the form
\[
E_1^{p,q} = \HM^{p+q}(F_{\lambda_p,\mu}, i_p^! \PC) \Rightarrow \HM^{p+q}(F_{\mu}, i^! \PC)
\]
We may express $i_p$ as the composition
\[
F_{\lambda_p, \mu} \hookrightarrow X_{\lambda_p} 
\stackrel{i_{\lambda_p}}{\hookrightarrow} X \]
where $F_{\lambda_p, \mu}$ is a smooth subvariety of $X_{\lambda_p}$. It follows that, if $d$ is the
(complex) codimension of $F_{\lambda_p, \mu}$ in $X_{\lambda_p}$, we have
\[
i_p^! \PC [2d] \cong (i_{\lambda_p}^! \PC)_{| F_{\lambda_p, \mu}}.
\]
(The isomorphism $i_p^![2d] \FC \cong i_p^* \FC$
is valid for any complex of sheaves on $X_{\l_p}$ whose
cohomology sheaves are local systems, as follows easily from the
smoothness of $F_{\lambda_p, \mu}$ and $X_{\l_p}$.)
As $\PC$ is $!$-even by assumption, $i_p^! \PC$ is isomorphic to a direct sum of local systems
in even degrees, all obtained by restriction from torsion free local systems on $X_{\lambda_p}$.

By assumption, the 
cohomology of $F_{\lambda_p, \mu}$ with values in such local systems is free
and concentrated in even degree and so the above spectral sequence
degenerates for parity reasons, whence the claim.
\end{proof}

One practical application of the previous result is that the existence of
parity sheaves follows from the existence of even resolutions:

\begin{cor}
Suppose that there exists an even, proper morphism
\[
\pi : \widetilde{X_{\lambda}} \to \ov X_{\lambda} \subset X
\]
which is an isomorphism over $X_{\lambda}$.

Assume there exists a $\natural$-parity complex $\widetilde{\PC}$ on
$\widetilde{X_{\lambda}}$ whose restriction to
$\pi^{-1}(X_\lambda)\cong X_\lambda$ is isomorphic to a shifted
indecomposable local system $\LC[d_\lambda]$ on $X_\lambda$. Then
there exists a $\natural$-parity sheaf $\PC$ on $X$ satisfying
\begin{enumerate} 
\item $\supp \PC = \ov X_{\lambda}$;
\item $\PC_{| X_{\lambda}} = \LC[d_\lambda]$.
\end{enumerate}

In particular, if $\pi$ is a resolution of singularities, then
the above holds for $\LC = \uk_\lambda$, since
in this case $\uk_{\widetilde{X_{\lambda}}}[d_\lambda]$ is parity.
\end{cor}

\begin{proof} By the previous proposition the direct image $\pi_*
  \widetilde{\PC}$ is a parity complex. By proper base change we 
  have that $(\pi_* \widetilde{\PC})_{|X_\l} = \LC[d_\l]$ and $\supp
  \pi_* \widetilde{\PC}= \ov{X_\l}$. Now decompose $\pi_*
  \widetilde{\PC}$ into indecomposable objects, and let $\PC$ denote the
  (necessarily unique) direct summand whose restriction to $X_\l$ is
  non-zero. Then
  $\PC$ is a parity sheaf with $\PC_{|X_\l} \cong \LC[d_\l]$ and
  $\supp \PC = \ov{X_\l}$ as claimed.
 \end{proof}

\subsection{Modular reduction of parity sheaves}
\label{subsec:modular}

Let $k\to k'$ be a ring homomorphism. In this section, we will
consider the behaviour of parity sheaves under the extension of
scalars functor, which we denote by
\[ k'(-):= k' \stackrel{L}{\otimes}_k - : D(X;k) \to D(X;k') \]

\begin{lem} \label{lem:parity reduc}
Suppose that $\FC\in D(X,k)$ is $?$-even (resp. odd), then $k'(\FC)$
is $?$-even (resp. odd). In particular, if $\FC$ is a
parity complex, then so is $k'(\FC$).
\end{lem}

\begin{proof}
It suffices to prove the $?$-even case. It is equivalent
to show that the $?$-restriction of $k'(\FC)$ to each point is
even. For any complex $\FC \in D(X; k)$ we have isomorphisms
$ i^? (k'(\FC)) \cong k'(i^? \FC) $
for $i$ the inclusion of a point (this follows, for example, from
Propositions 2.3.5 and 2.5.13 of \cite{KS1}). 
By definition, if $\FC$ is $?$-even then the cohomology of
$i^? \FC$ vanishes in odd degrees and is free and therefore
flat.  Thus 
$ k' (i_{\lambda}^?\FC) = k' \otimes i_\lambda^?\FC $
 and $k'(\FC)$ is $?$-even.
\end{proof}

We will now restrict our attention to the case when $k=\OM$ and $k'=\FM$
and $k\to k'$ is the residue map. Recall that $\OM$ denotes a complete
discrete valuation ring and $\FM$ its residue field. We assume that 
\eqref{assump-parity} and \eqref{assump-parity2} holds for $D(X, \OM)$.

In this case, $k'(-)$ is the modular reduction functor:
\begin{align*}
\FM(-) := \FM \stackrel{L}{\otimes}_\OM - & : D(X,\OM) \to D(X, \FM)
\end{align*}

First we claim that in this situation, the implication of the previous
theorem is in fact an equivalence.

\begin{prop}
\label{prop-reduction}
A complex $\FC\in D(X;\OM)$ is $?$-even (resp, $?$-odd or
 parity) if and only if $\FM\FC$ is.
\end{prop}

\begin{remark}
This proposition is analogous to \cite[Prop. 42(a)]{Serre}, which states,
for a finite group $G$, an $\OM[G]$-module which is free as an
$\OM$-module is projective if and only if
its reduction is a projective $\FM[G]$-module.
\end{remark}

\begin{proof}
Having proved ``only if'' it remains to prove ``if''.

Again, it suffices to check the $?$-restrictions to points. As before, we have 
$ i^? (\FM\FC) \cong \FM(i^?\FC).$
This time, we wish to show that if $i^? (\FM\FC)$ vanishes in odd
degrees, then $i^? \FC$ does too and is a free $\OM$-module. 

The derived category over a point $D(\pt;\OM)$ (resp. $D(\pt;\FM)$) is
equivalent to the bounded derived category of finitely generated $\OM$-modules,
$D(\Mod_\OM)$ (resp. finite dimensional $\FM$-vector spaces, $D(\Vect_\FM)$).  The ring $\OM$ is hereditary,
which implies that any object in $D(\Mod_\OM)$ is 
isomorphic to its cohomology.
Using this it is easy to see that if the cohomology of $i^?\FC$ had torsion, then $\FM
i^?\FC$ would have non-trivial cohomology concentrated in two consecutive degrees.
Hence $i^?\FC$ is a free $\OM$-module and is even.
\end{proof}

\begin{prop}
If $\EC\in D(X,\OM)$ is a parity sheaf, then $\FM\EC$
is also a parity sheaf.  In other words, for any $\LC
\in \Loc(X_\lambda,\OM)$, we have
\[\FM\EC(\lambda,\LC)\cong \EC(\lambda,\FM\LC).\]
\end{prop}

\begin{proof} For local systems $\LC, \LC^{\prime}\in \Loc(X_\lambda,\OM)$ on $X_{\lambda}$,
we have
\begin{equation}\label{eq-modred}
 \FM \otimes \Hom(\LC, \LC^{\prime}) \cong \Hom(\FM \LC, \FM
 \LC^{\prime}) 
\end{equation}
Now consider $\FC$ (resp. $\GC$) in $D(X, \OM)$
which is $*$- (resp. $!$-) parity. Then using the proof of Proposition
\ref{prop-parityhom} and ~\eqref{eq-modred} above 
for $\LC=i^*_\lambda
\FC,\LC'=i^!_\lambda \GC$,
we see that the natural morphism yields an isomorphism:
\begin{equation}
\label{eq:mod hom parity}
\FM \otimes \Hom(\FC, \GC) \stackrel{\sim}{\to} \Hom(\FM \FC,
\FM \GC).
\end{equation}
Finally, let $\FC=\GC=\EC(\l, \LC) \in D(X;\OM)$ 
be a parity
sheaf. It follows that we have a surjection
\[
\End(\EC(\l, \LC)) \twoheadrightarrow \End(\FM \EC(\l, \LC)).
\]
It follows by Lemma \ref{lem-quotient local} that $\End(\FM \EC(\l, \LC))$
is a local ring, and hence $\FM \EC(\l, \LC)$ is indecomposable. We
also know that $\FM \EC(\l, \LC)$ is a parity complex by Lemma \ref{lem:parity reduc}. Hence
we have an isomorphism $\EC(\l, \FM \LC) \cong \FM \EC(\l, \LC)$ by
Theorem \ref{indecs}.
\end{proof}

\begin{remark}
This is a partial analogue to \cite[Prop. 4.2(b)]{Serre}, which states
that for each projective $\FM[G]$-module $F$ there exists a unique (up to
isomorphism) projective $\OM[G]$-module whose reduction is isomorphic
to $F$.
\end{remark}

\begin{prop} \label{prop:almostall}
If $\LC \in \Loc(X_\l,\ZM)$ is a local system on a stratum of $X$ such that $\ic(\l,\QM\LC) \cong \EC(\l,\QM\LC)$ is a parity sheaf, then $\ic(\l,\FM_p \LC) \cong \EC(\l,\FM_p \LC)$ for all but finitely many primes $p$.
\end{prop}

\begin{proof}
Suppose there is no $p$-torsion in the cohomology of the stalks and costalks of $\ic(\l,\LC)$ for some prime $p$.  Then the graded dimensions of the stalks and costalks of $\QM\ic(\l,\LC) ( \cong \ic(\l,\QM\LC) )$ and $\FM_p \ic(\l,\LC)$ coincide.  It follows that $\FM_p \ic(\l,\LC)$ is isomorphic to $\ic(\l,\FM_p\LC)$ and, by the parity assumption on $\ic(\l,\QM\LC)$, is a parity sheaf.

It remains to show that the cohomology groups of the stalks and
costalks of $\ic(\l,\LC)$ contain torsion for only finitely many prime
numbers.  This is true because they are finitely generated
$\ZM$-modules and are non-zero in finitely many degrees and on finitely many strata.
\end{proof}

\begin{remark}
In general, it is very difficult to determine for which $p$ the conclusion of the previous proposition holds.
\end{remark}

\subsection{Torsion primes} \label{sec-torsion}

Our assumptions~(\ref{assump-parity}) and~(\ref{assump-parity2}) on
the space $X$ are quite strict.  If we work in the equivariant
setting, they might not even be satisfied when $X$ is a single
point. However, once we invert a set of prime numbers in $k$ depending
on the group $G$, called the torsion primes, the conditions are
satisfied at least for a point. 
In this subsection,
we recall from \cite{Steinberg-torsion} some facts about torsion primes.

Let $G$ be a reductive group with associated root datum $(\XB,\Phi,\YB,\Phi^\vee)$.
A reductive subgroup of $G$ is called regular if it contains a maximal
torus $T$ of $G$. Because our base field is $\CM$, regular reductive
subgroups containing $T$ are in bijection with $\ZM$-closed subsystems of
$\Phi$, i.e. $\Phi_1 \subset \Phi$ satisfying $\ZM \Phi_1 \cap \Phi = \Phi_1$.

\begin{defi}
A prime $p$ is a {\bf torsion prime} for $G$ if $\pi_1(H) = \YB/\ZM\Phi_1^\vee$ has $p$-torsion,
for some regular reductive subgroup $H$ of $G$ with root datum $(\XB,\Phi_1,\YB,\Phi_1^\vee)$.
\end{defi}

It follows from the definition that the torsion primes of any regular
reductive subgroup $H$ of $G$ are among those of $G$. 
The reason that we are interested in torsion primes is the 
following theorem of Borel \cite{Borel-torsion,Rothenberg-Steenrod,DEM,Kac-torsion}.

\begin{thm}
\label{thm-BGtorsion}
The following conditions are equivalent:
\begin{enumerate}
\item the prime $p$ is not a torsion prime for $G$;
\item the cohomology $H^*(G,\ZM)$ of $G$ has no $p$-torsion;
\item the cohomology $H^*(B_G,\ZM)$ of the classifying space of $G$ has no $p$-torsion.
\end{enumerate}
Moreover, if $k$ is a field whose characteristic either is zero or satisfies the
above conditions, then $H^*(G,k)$ is an exterior algebra with generators
of odd degrees, while $H^*(B_G,k)$ is a polynomial
algebra on generators of one higher degree.
\end{thm}

By the universal coefficient theorem, we can conclude that if $k$ is a ring in which all torsion primes for $G$ are
invertible, then $H^*_G(\pt,k)$ is even and torsion-free, and the same is true
for any regular reductive subgroup $H$ of $G$.

We now recall how to determine the set of torsion primes.
A reductive group has the same torsion primes as its derived subgroup.
The torsion primes of a semi-simple group
are those of its simply connected cover, together with the primes
dividing the order of its fundamental group. The set of torsion primes of a
semi-simple and simply connected group is the union of those of its
simple factors.

Hence we are reduced to considering $G$ simple and simply connected.
Let us choose a system of simple roots $\Delta$ of $\Phi$.
Let then $\tilde\alpha$ denote the highest root of $\Phi$,
and let $\tilde\alpha^\vee = \sum_{\alpha\in\Delta} n_\alpha^\vee \alpha^\vee$
be the decomposition of the corresponding coroot into simple coroots.
Finally, let $n^\vee$ denote the maximum of the $n_\alpha^\vee$.

\begin{thm}
 If $G$ is simple and simply connected and $p$ is a prime, then the following conditions are equivalent:
\begin{enumerate}
 \item $p$ is a torsion prime for $G$;
 \item $p \leq n^\vee$;
 \item $p$ is one of the $n_\alpha^\vee$;
 \item $p$ divides one of the $n_\alpha^\vee$.
\end{enumerate}
\end{thm}

Thus the torsion primes of the simple, simply connected groups are given by the following table:

\[
 \begin{array}{c|c|c|c}
  A_n, C_n & B_n (n \geq 3), D_n, G_2 & E_6, E_7, F_4 & E_8\\
\hline
\text{none} & 2 & 2,3 & 2,3,5
 \end{array}
\]

\subsection{Ind-varieties}
\label{sec:ind}

In this section we comment on how the results of this section
generalise straightforwardly to the  
slightly more general setting of ind-varieties. Recall that 
an ind-variety $X$
is a topological space, together with a filtration
\[ X_0 \subset X_1 \subset X_2 \subset \dots \]
such that $X = \cup X_n$,
each $X_n$ is a complex algebraic variety, and the
inclusions $X_n \hookrightarrow X_{n+1}$ are closed embeddings.
We will always assume that each $X_n$ carries the classical
topology and equip $X$ with the final topology
with respect to all inclusions $X_n \hookrightarrow X$.
By a stratification of $X$ we mean a stratification of each
$X_n$ such that the inclusions preserve the strata. We will 
also consider the case where $X$ is acted upon by a
linear algebraic pro-group $G$, by which we mean that
$G$ acts on each $X_n$ through a quotient isomorphic to a linear algebraic group, that each such action is algebraic, and that each of the inclusions
$X_n \hookrightarrow X_{n+1}$ are $G$-equivariant. For the 
basic properties of ind-varieties and pro-groups we refer
the reader to \cite[Chapter 4]{Ku}.

Now fix a complete local ring $k$ and let 
\[ X = \bigsqcup_{\lambda \in \Lambda} X_{\lambda}
 \]
be a stratified ind-variety, or an 
ind-$G$-variety with $G$-stable stratification for some linear
algebraic group $G$.
We write $D(X)$ for the full subcategory of the 
bounded (equivariant) derived category of sheaves of $k$-vector 
spaces consisting of objects $\FC$ such that:
\begin{enumerate}
\item the support of $\FC$ is contained in $X_n$ for some $n$;
\item the cohomology sheaves of $\FC$ are constructible with respect to the 
stratification.
\end{enumerate}

We assume that \eqref{assump-parity} holds for the strata of $X_{\lambda}$.
The notion of parity still makes sense 
and it is immediate that the analogue of
Theorem \ref{indecs} applies. 
In particular, given any (equivariant) indecomposable local system 
$\LC \in \Loc(X_{\lambda})$ there is, up to isomorphism and shifts,
at most one indecomposable 
parity sheaf $\EC(\lambda) \in D(X)$ supported on $\ov X_{\lambda}$
and extending $\LC[d_{\lambda}]$.

In what follows we will
refer without comment to results which we have proved previously
for varieties, but where an obvious analogue holds for ind-varieties.

\section{The Decomposition Theorem and intersection forms}
\label{sec-decomp}

In their proof of the decomposition theorem for semi-small maps
\cite{dCM}, de Cataldo and
Migliorini highlighted the crucial role played by 
intersection forms associated to the strata of the target: 
a certain splitting implied by the decomposition theorem is equivalent to 
these forms being non-degenerate. Then they prove
the non-degeneracy using techniques from Hodge theory.

In our situation, where we consider modular coefficients, these forms
may be degenerate. In this section we explain how the non-degeneracy
of these forms, together with the semi-simplicity of certain local
systems, provide necessary and sufficient conditions for the
Decomposition Theorem to hold in positive characteristic.
For this, we do not have to assume that $X$ satisfies
\eqref{assump-parity} or \eqref{assump-parity2}.
If the map we consider is even (and not necessarily semi-small) then the direct image will
be a direct sum of parity sheaves. Assuming now the parity conditions
on the strata (ensuring the uniqueness of parity sheaves),
we will see that even if the decomposition fails, one can still use
intersection forms to determine the multiplicities of parity sheaves
that occur in the direct image of the constant sheaf.

In Section \ref{subsec-intproduct} we recall the definition and basic
properties of intersection forms on Borel-Moore homology.
In Section \ref{subsec-DTSS} we relate these intersection forms to the
failure of the Decomposition Theorem for semi-small maps in
characteristic $p$. In Section \ref{subsec-decompparity} we examine
the decomposition of the direct image into parity sheaves in the case
where the morphism is even.

Aside from Section \ref{subsec-intproduct}, we assume that $k$ is a
field (but see Remark \ref{rem:Odecomp}).
All cohomology groups are assumed
to have values in $k$ and, as always,
dimension refers to the complex dimension unless
otherwise stated.

\subsection{Borel-Moore homology and intersection forms}
\label{subsec-intproduct}

In this subsection we recall some basic properties of Borel-Moore
homology and intersection forms. For more details the reader is
referred to \cite{Ful} or \cite{CG}.

For any variety $X$ we let $a_X : X \to \pt$ denote the projection to
a point. The dualising sheaf on $X$ is
$\omega_X := a_X^! \uk_{\pt}$. One may define the Borel-Moore homology
of $X$ to be
\begin{equation*}
\HBM_i(X) = H^{-i}( a_{X*} a_X^! \uk_{\pt}) = \Hom( \uk_X, \omega_X[-i]).
\end{equation*}
A proper map $\pi : X \to Y$ induces a map $\HBM_\bullet(X) \to
\HBM_\bullet(Y)$ which may be described as follows: given a class $\a :
\uk_X \to \omega_X[-i] \in \HBM_i(X)$ its image in $\HBM_i(Y)$ is the class
\begin{equation*}
\uk_Y \to \pi_*\pi^*\uk_Y = \pi_* \uk_X \stackrel{\pi_* \a}{\longrightarrow}
\pi_* \omega_X[-i] = \pi_!\pi^! \omega_Y[-i] \to \omega_Y[-i]
\end{equation*}
where the first and last arrows are the adjunction morphisms.

Let $Y$ be a smooth and connected variety of dimension $n$. As $Y$ has
a canonical orientation (after choosing an orientation on $\CM$) we
have an isomorphism $\mu_Y : \uk_Y \stackrel{\sim}{\to}
\omega_Y[-2n]$. If we regard $\mu_Y$ as an element of $\HBM_{2n}(Y)$
it is called the {\bf fundamental class}. Even if $Y$ is singular 
 of dimension $n$, $\HBM_{2n}(Y)$ is still freely
generated by the fundamental classes of the irreducible components of
$Y$ of maximal dimension.

Now suppose that $F \stackrel{i}{\hookrightarrow} Y$ is a closed
embedding of a variety $F$ into a smooth variety of dimension $n$.
For all $m$ we have a canonical isomorphism
\[ \HBM_m(F) \cong H^{2n-m}(Y, Y - F). \]
Recall that there exists a cup product on relative cohomology. We may
use this to define an intersection form of $F$ inside $Y$:
\begin{equation*}
\xymatrix@C=0.1cm{
\HBM_{p}(F) \ar[d]^{\sim}& \otimes &  \HBM_{q}(F)\ar[d]^{\sim} & \ar[rr]  &  & &   
\HBM_{p + q - 2n}(F)\ar[d]^{\sim} \\
H^{2n - p}(Y, Y- F) & \otimes & H^{2n - q}(Y, Y- F) & \ar[rr]^{\cup} &  & & H^{4n - p - 
q}(Y, Y- F)}
\end{equation*}
Note that this product depends on the inclusion $F
\hookrightarrow Y$.

Below we will need a slight variation of this intersection form in case $p + q = 2n$.
If $F$
is a proper subvariety of $Y$ then we can compose the intersection
form with the map $\HBM_0(F) \to \HBM_0(\pt)$ induced from the proper
map $F \to \pt$ to obtain a pairing
\[
B_F^m: \HBM_{n+m}(F) \times   \HBM_{n-m}(F) \to  
\HBM_0(\pt) = k.
\]
Geometrically, this pairing corresponds to intersecting cycles
but forgetting in which connected component of $F$ the points live.

It is particularly interesting when $F$ is proper and
half-dimensional inside $Y$. In this case we obtain an intersection form
\begin{equation*}
B_F^0: \HBM_{\top}(F) \otimes \HBM_{\top}(F) \to k
\end{equation*}
where $\top$ denotes the real dimension of $F$. From the above
comments, $\HBM_{\top}(F)$ has a basis given by the irreducible
components of maximal dimension of $F$. It also follows that this
intersection form over any ring is obtained by extension of
scalars from the corresponding form over $\ZM$.

The effect of forming the Cartesian product with a smooth and contractible space on
Borel-Moore homology is easy to describe (and will be needed
below). If $U$ is an open contractible subset
of $\CM^m$, then for any $i \in \ZM$, we have canonical isomorphisms
\begin{align*}
\HBM_{i}(X \times U) \cong \HBM_{i-2m}(X).
\end{align*}
These isomorphisms are compatible with the intersection forms of $F
\hookrightarrow Y$ and $F \times U \hookrightarrow Y \times U$.
\excise{
We have a commutative diagram with vertical isomorphisms
\begin{equation*}
\xymatrix@C=0.1cm{
\HBM_{p}(F) \ar[d]^{\sim}& \otimes &  \HBM_{q}(F)\ar[d]^{\sim} & \ar[rr]  &  & &   
\HBM_{p + q - 2n}(F)\ar[d]^{\sim} \\
\HBM_{p+2m}(F \times \RM^m) & \otimes &  \HBM_{q + 2m}(F\times \RM^m)  & \ar[rr]  
&  & &   \HBM_{p + q - 2n + 2m}(F \times \RM^m)}
\end{equation*}}

\subsection{Multiplicities and intersection forms}
\label{subsec-DTSS}

In this Section, we do not assume that the stratified variety $X$
satisfies the conditions \eqref{assump-parity} or \eqref{assump-parity2}.
Moreover, in this Section and the next, we assume that $k$ is a field.

In Section \ref{subsub-KS} we explain how multiplicities of indecomposable
objects in a $k$-linear Krull-Remak-Schmidt category are encoded in certain
bilinear forms.
In Section \ref{subsub-sspoint} we examine the splitting at the ``most
singular point'' and relate it to an intersection form on the fibre.
This is used in Section \ref{subsub-dt} to relate the
Decomposition Theorem to the non-degeneracy of the intersection forms
attached to each stratum.

\subsubsection{Bilinear forms and multiplicities in Krull-Remak-Schmidt categories}

\label{subsub-KS}

Let $H$ and $H^{\prime}$ be finite dimensional $k$-vector spaces and
consider a bilinear map
\[
B : H \times H^{\prime} \to k.
\]
We define:
\begin{gather*} 
{}^{\perp} B := \{ \alpha \in H \; | \; B(\alpha, \beta) = 0 
\text{ for all $\beta \in H^{\prime}$} \}, \\
B^{\perp} := \{ \beta \in H^{\prime} \; | \; B(\alpha, \beta) = 0 
\text{ for all $\alpha \in H$} \}.
\end{gather*}
Then $B$ induces a perfect pairing
\[
H / {}^{\perp} B \times H^{\prime} / B^{\perp} \to k
\]
and we have equalities
\[ \dim (H / {}^{\perp} B)
 = \rank B = 
\dim (H^{\prime} / B^{\perp}). \]
If $H = H^{\prime}$ and $B$ is a symmetric bilinear form then we write $\rad
B$ instead of ${}^{\perp} B = B^{\perp}$.

Let $\CC$ be a Krull-Remak-Schmidt $k$-linear category with finite
dimensional $\Hom$-spaces (see Section \ref{subsec-naa}). Let $a \in \CC$
denote an indecomposable object. Given any object $x \in \CC$
we can write $x \simeq a^{\oplus m} \oplus y$ such that $a$ is not 
a direct summand in $y$. The integer $m$ is called the {\bf multiplicity}
of $a$ in $x$. This multiplicity is well-defined because $\CC$ is
Krull-Remak-Schmidt.

\begin{lem} \label{lem-multiplicity} Assume that $\End(a) = k$.  Composition gives us a pairing:
\begin{equation}
\label{eq-KSPairing}
\begin{array}{ccccc}
B &:& \Hom(a,x) \times \Hom(x,a) &\longto& \End(a) = k\\
&&(\alpha,\beta) &\longmapsto& \beta \circ \alpha.
\end{array}
\end{equation}
The multiplicity of $a$ in $x$ is equal to the rank of $B$.
\end{lem}

\begin{proof}
Choose an isomorphism $\phi : x \elem{\sim} a^{\oplus m} \oplus y$.
Inclusion and projection define subspaces
$\Hom(a,y)  \subset \Hom(a,x)$ and $\Hom(y,a) \subset \Hom(x,a)$.
We will show that
\begin{equation}
\label{eq:hom=rad}
\Hom(a,y) = {}^{\perp} B \qquad \text{and} \qquad \Hom(y,a) = B^{\perp}.
\end{equation}
Thus these subspaces do not depend on the choice of $\phi$, and
this will show that the multiplicity of $a$ in $x$ is equal to the rank of $B$.

The isomorphism $\phi$ induces isomorphisms
\begin{gather*}
\Hom(a,x) \simeq \Hom(a,a^{\oplus m}) \oplus \Hom(a,y),\\
\Hom(x,a) \simeq \Hom(a^{\oplus m},a) \oplus \Hom(y,a).
\end{gather*}
For $\alpha$ in $\Hom(a,x)$ and $\beta$ in $\Hom(x,a)$, we write
$\alpha = \alpha_1 \oplus \alpha_2$ and
$\beta = \beta_1 \oplus \beta_2$ for the corresponding decompositions.
Thus we have $B(\alpha,\beta) = \beta_1 \alpha_1 + \beta_2 \alpha_2$.

We actually have
\begin{equation}
\label{eq:b2a2=0}
\beta_2 \alpha_2 = 0.
\end{equation}
Otherwise, since $\End(a) = k$, we could
assume that $\beta_2 \alpha_2 = \id$, in which case
$e := \alpha_2 \beta_2$ would be an idempotent in $\End(y)$, and $a$ would be a
direct summand of $y$, contrary to our assumption. 

Thus we have $B(\alpha,\beta) = \beta_1 \alpha_1$ for all
$\alpha$ and $\beta$. Hence we have inclusions
\begin{equation}
\label{eq-incl}
\Hom(a,y) \subset {}^{\perp} B \qquad \text{and} \qquad \Hom(y,a) \subset B^{\perp},
\end{equation}
and $B$ induces a bilinear form
\[
\widetilde{B} : \Hom(a,a^{\oplus m}) \times \Hom(a^{\oplus m},a) \longto \End(a) = k.
\]
Now
$\Hom(a,a^{\oplus m}) \simeq \End(a)^{\oplus m} \simeq k^{\oplus m}$
and similarly $\Hom(a^{\oplus m},a) \simeq k^{\oplus m}$.
With these identifications $\widetilde{B}$ is just the standard bilinear form
on $k^{\oplus m}$, hence it is non-degenerate, and we have equalities
in \eqref{eq-incl}.
\end{proof}

Let us assume further that $\CC$ is equipped with a duality\footnote{A
  duality is a contravariant equivalence $D : \CC \stackrel{\sim}{\to}
  \CC^{op}$ together with a natural isomorphism $D^2 \cong Id_{\CC}$,
  where $Id_{\CC}$ denotes the identity functor on $\CC$.} $D$ 
and we have isomorphisms $a \stackrel{\sim}{\to} Da$ and $x
\stackrel{\sim}{\to} Dx$. Then, using these isomorphisms, we may
identify $\Hom(a,x)$ and $\Hom(x,a)$. In which case the composition
\eqref{eq-KSPairing} is given by a bilinear form on $H = \Hom(a,x) =
\Hom(x,a)$ and the multiplicity of $a$ in $x$ is equal to the rank of
this form.

\subsubsection{Splitting at the most singular point}
\label{subsub-sspoint}
Consider a proper surjective morphism
\[ \pi : \widetilde{X} \to X \]
with $\widetilde{X}$ smooth of dimension $n$. Fix a point $s \in X$ and form the Cartesian diagram:
\[
\begin{array}{c}
\xymatrix@C=1cm{ F \ar[r]^{\tilde{i}} \ar[d]^{\tilde{\pi}} & \widetilde{X} \ar[d]^{\pi}  \\
\{ s \} \ar[r]^{i} & X }
\end{array} \]
Let $B_F^m$ denote the intersection form
\[
B_F^m: \HBM_{n+m}(F) \times \HBM_{n-m}(F) \to \HBM_0(\pt) = k
\]
associated to the inclusion $F \hookrightarrow \widetilde{X}$.\footnote{See
Section \ref{subsec-intproduct} for the definition of the intersection
form. Note in particular
that the image of $B_F^m$ is $\HBM_0(\pt)$ and not $\HBM_0(F)$.}

\begin{prop} \label{prop-sspointmult}
The multiplicity of $i_*\uk_{s}[m]$ as a direct summand of
$\pi_* \uk_{\widetilde{X}}[n]$ is equal to the rank of $B_F^m$. 
\end{prop}

\begin{remark} Because  $\pi_* \uk_{\widetilde{X}}[n]$ is self-dual
 $i_*\uk_{s}[m]$ and $i_*\uk_{s}[-m]$ occur with equal multiplicities as
direct summands of $\pi_* \uk_{\widetilde{X}}[n]$. This also follows
from the above proposition: $B_F^{m}$ and $B_F^{-m}$ are
transpose, and hence have the same rank.
\end{remark}

\begin{proof} By the discussion of the previous section, if $B$ denotes
  the pairing given by composition
\[ B : \Hom(i_*\uk_{s}[m], \pi_* \uk_{\widetilde{X}}[n])
\times \Hom(\pi_* \uk_{\widetilde{X}}[n], i_*\uk_{s}[m]) \to k \]
then the multiplicity of $i_* \uk_s[m]$ in $\pi_* \uk_{\widetilde{X}}[n]$
is given by the rank of $B$.
A string of adjunctions gives canonical identifications:
\begin{gather*} 
\Hom(i_*\uk_{s}[m], \pi_* \uk_{\widetilde{X}}[n]) = \HBM_{n+m}(F), \\
\Hom(\pi_* \uk_{\widetilde{X}}[n], i_*\uk_{s}[m]) = \HBM_{n-m}(F).
\end{gather*}
Hence we are interested in a pairing
\begin{equation} \label{pair}
B: \HBM_{n+m}(F) \times \HBM_{n-m}(F) \to k.
\end{equation}
By the lemma below, this is the intersection form $B_F^m$.
The proposition then follows.
\end{proof}

\begin{lem} The pairing $B$ in \eqref{pair} is the intersection form $B_F^m$. \end{lem}

\begin{proof} 
 In the course of the following proof all morphisms which are not
  described explicitly are either adjunction morphisms or the
  canonical morphisms $i^! \to i^*$ and $\tilde{i}^! \to
  \tilde{i}^*$. Also, every time we say ``identify'' we mean 
  ``canonically identify''.

Recall that the intersection form is defined via the cup product on 
relative cohomology. Let $\gamma_1$ and $\gamma_2$ be classes
in $\HBM_{n+m}(F) = H^{n-m}(\widetilde{X}, \widetilde{X}- 
F)$ and $\HBM_{n-m}(F) = H^{n+m}(\widetilde{X}, \widetilde{X}- 
F)$ respectively. If we represent them as morphisms $
\gamma_1: \uk_{\widetilde{X}} \to \tilde{i}_! \tilde{i}^{!}
\uk_{\widetilde{X}}[n-m]$ and $\gamma_2  : \uk_{\widetilde{X}} \to \tilde{i}_! \tilde{i}^{!} \uk_{\widetilde{X}}[n+m]$ their cup product is the morphism
\begin{equation*}
\gamma_1 \cup \gamma_2 : \uk_{\widetilde{X}} \stackrel{\gamma_1}{\longrightarrow} \tilde{i}_! \tilde{i}^{!} \uk_
{\widetilde{X}}[n-m]  \to \uk_{\widetilde{X}}[n-m] \stackrel{\gamma_2[n-m]}{\longrightarrow} \tilde{i}_! \tilde{i}^{!} \uk_{\widetilde
{X}}[2n].
\end{equation*}
Because $\widetilde{X}$ is a smooth variety the
  fundamental class gives an isomorphism 
  $\mu_{\widetilde{X}} : \uk_{\widetilde{X}} \stackrel{\sim}{\to}
  \omega_{\widetilde{X}}[-2n]$ (see \ref{subsec-intproduct}).  Using $i^!
  \omega_{\widetilde{X}} = \omega_F$  and
  adjunction we have an
  identification
\[
\Hom(\uk_{\widetilde{X}}, \tilde{i}_! \tilde{i}^{!}
\uk_{\widetilde{X}}[n-m]) = \Hom(\uk_F, \omega_F[-n-m]).
\]
One can now check that the intersection form
\[
\HBM_{n+m}(F) \times \HBM_{n-m}(F) \to \HBM_0(F)
\]
can be described as follows: given classes $\a : \uk_F \to 
\omega_F[-n-m]$ and $\b : \uk_F \to \omega_F[-n+m]$, their pairing
under the intersection form is the class
\[
\uk_F \to \tilde{i}^!\omega_{\wt{X}}[-n-m] = \tilde{i}^!
\uk_{\wt{X}}[n-m] \to \tilde{i}^* \uk_{\wt{X}}[n-m] = \uk_F[n-m] \to
\omega_F. 
\]

Now consider the following diagram:
\[
\small{
\xymatrix@C=0.3cm@R=0.5cm{
\Hom(\tilde{\pi}^* \uk_{\{ s \}}[m], \tilde{i}^!
\uk_{\widetilde{X}}[n]) \otimes
\Hom(\tilde{i}^*\uk_{\widetilde{X}}[n], \tilde{\pi}^!
\uk_{\{s \}}[m]) \ar[r]^(0.85){B_1} \ar[d]^{\textrm{a}}
& H \ar[d]^{\phi} \\
\Hom(\uk_{\{ s \}}[m], \tilde{\pi}_*\tilde{i}^! \uk_{\widetilde{X}}[n]) \otimes 
\Hom( \tilde{\pi}_*\tilde{i}^* \uk_{\widetilde{X}}[n], \uk_{\{ s \}}[m])
\ar[r]^(0.78){B_2}  \ar[d]^{\textrm{pbc}} & \End(
\uk_{\{ s \}})  \ar[d]^=\\
\Hom(\uk_{\{ s \}}[m], i^!\pi_* \uk_{\widetilde{X}}[n]) \otimes 
\Hom( i^* \pi_* \uk_{\widetilde{X}}[n], \uk_{\{ s \}}[m])
\ar[r]^(0.78){B_3} \ar[d]^{\textrm{a}} & \End(
\uk_{\{ s \}})  \ar[d]^{i_*}\\
\Hom(i_* \uk_{\{ s \}}[m], \pi_* \uk_{\widetilde{X}}[n]) \otimes 
\Hom( \pi_* \uk_{\widetilde{X}}[n], i_* \uk_{\{ s \}}[m])
\ar[r]^(0.78){B_4} & \End(i_*
\uk_{\{ s \}})
} }
\]
where:
\begin{enumerate}
\item $H := \Hom(\tilde{\pi}^* \uk_{\{ s \}}[m], \tilde{\pi}^!
\uk_{\{s \}}[m]) = \Hom(\uk_F[m],\omega_F[m])$,
\item each pairing $B_i$ is canonical,
\item the arrows labelled ``a'' are obtained via adjunction,
\item the arrow labelled ``pbc'' is obtained via proper base change,
\item the morphism $\phi$ is the map which sends $\gamma : \tilde{\pi}^*
  \uk_{\{ s \}}[m] \to \tilde{\pi}^! \uk_{\{s \}}[m]$ to the morphism
\[
\phi(\gamma) : \uk_{\{ s \}}[m] \to \tilde{\pi}_*\tilde{\pi}^*
  \uk_{\{ s \}}[m] \stackrel{\tilde{\pi}_* \gamma}{\longrightarrow} \tilde{\pi}_*\tilde{\pi}^! \uk_{\{s \}}[m] \to \uk_{\{ s \}}[m].
\]
\end{enumerate}

It is straightforward to check that this diagram is
commutative.\footnote{One needs that the following diagram of functors
  is commutative (where the vertical arrows are the base-change
  isomorphisms):
\[ \xymatrix{
i^! \pi_* \ar[r] \ar[d] & i^* \pi_* \ar[d] \\
\tilde{\pi}_*\tilde{i}^! \ar[r] & \tilde{\pi}_* \tilde{i}^*
}
\]
This can be checked directly, by first checking the statement on
sheaves.
}

By the above discussion, $B_1$ may be identified with the
intersection form
\[
\HBM_{n+m}(F) \times \HBM_{n-m}(F) \to H = \HBM_0(F).
\]
Also, if we identify $H = \HBM_0(F)$ and $\End(\uk_{\{
  s\}}) = \HBM_0(\pt)$ then $\phi$ is the map $\HBM_0(F) \to
\HBM_0(\pt)$ induced from the projection $F \to \pt$.

It follows that if we identify
\begin{gather*} 
\Hom(i_*\uk_{\{s\}}[m], \pi_* \uk_{\widetilde{X}}[2n]) = \HBM_{n+m}(F), \\
\Hom(\pi_* \uk_{\widetilde{X}}[2n], i_*\uk_{\{s\}}[m]) = \HBM_{n-m}(F)
\end{gather*}
then we may identify $B_4$ with the composition
\[
\HBM_{n+m}(F) \otimes \HBM_{n-m}(F) \to \HBM_0(F) \to \HBM_0(\pt)
\]
which by definition is $B_F^m$.
\end{proof}

The proof of Proposition \ref{prop-sspointmult} has the following corollary:

\begin{cor} \label{cor-inducedmorph}
The natural morphism
\[ H^m( i^! \pi_* \uk_{\widetilde{X}}[n]) \to 
H^m(i^* \pi_* \uk_{\widetilde{X}}[n]) \]
may be canonically identified with the morphism
\[ \HBM_{n-m}(F) \to \HBM_{n+m}(F)^* \]
induced by the intersection form. \end{cor}

\subsubsection{A criterion for the Decomposition Theorem}
\label{subsub-dt}

Let $X$ be a connected, equidimensional variety 
equipped with an algebraic stratification into connected strata
\[ X = \bigsqcup_{\lambda \in \Lambda} X_{\lambda}. \]
In this subsection we do not make any parity assumptions on our stratification. 
We write $d_X$ for the dimension of $X$ and, as usual, 
write $d_{\lambda}$ and $i_{\lambda}$ for the dimension
and inclusion of $X_{\lambda}$ respectively. We fix a 
smooth variety $\widetilde{X}$ and a stratified, proper, surjective,
semi-small morphism
\[ f : \widetilde{X} \to X. \]
We want to understand when the perverse sheaf
$f_* \uk_{\widetilde{X}}[d_{\widetilde{X}}]$
decomposes as a direct sum of intersection cohomology sheaves.
This may be thought of as a global version of the previous section.

As we have assumed that the stratification of $X$ is algebraic
\cite[3.2.23]{CG}, at each point $x \in X_{\lambda}$, we
can choose a stratified slice $N_{\lambda}$ to $X_\l$ in $X$. We obtain 
a Cartesian diagram with $\widetilde{N_{\lambda}}$ smooth:
\begin{equation} \label{diag-gccart}
\begin{array}{c}
\xymatrix{ F_{\lambda} \ar[r]^{\tilde{i}} \ar[d]^{\tilde{\pi}} & \widetilde{N_{\lambda}} \ar[d]^{\pi}  \\
\{ x \} \ar[r]^{i} & N_{\lambda} }
\end{array}
\end{equation}
Note that, as $f$ is semi-small, the dimension of $F$ is less than
or equal to $\frac{1}{2}(d_X - d_{\lambda})$. If equality holds we say
that $X_{\lambda}$ is {\bf relevant} (see \cite{BM2}).

\begin{defi} The {\bf intersection form associated to 
$X_{\lambda}$} is the intersection form on $\HBM_{d_X-d_{\lambda}}(F_{\lambda})$
given by the inclusion $F_{\lambda} \hookrightarrow \widetilde{N_{\lambda}}$.
\end{defi}

Note that $\HBM_{d_X-d_{\lambda}}(F_\l)$ is non-zero if and only if
$X_{\lambda}$ is relevant. Using the discussion at the end of Section 
\ref{subsec-intproduct} it is straightforward to see  that the
above intersection form does not depend on the choice of $x$. 
For each $\lambda$, we define
\begin{equation} \label{eq:Llambda}
\LC_{\lambda} :=
\HC^{-d_{\lambda}} (i_{\lambda}^! f_* \uk_{\widetilde{X}}[d_{\widetilde{X}}]).
\end{equation}
Note that $\LC_{\lambda}$ is a local system on $X_{\lambda}$
which is non-zero if and only if $X_{\lambda}$ is
relevant. The aim of this subsection is to show:

\begin{thm} \label{thm-intformsgeneral}
Suppose that the intersection forms associated to all strata are
non-degenerate. Then one has an isomorphism:
\begin{equation*}
f_* \uk_{\widetilde{X}}[d_{\widetilde{X}}] \cong \bigoplus_{\lambda \in \Lambda} \ic( \LC_{\lambda}).
\end{equation*}
In that case, the full Decomposition Theorem holds if and only if each local
system $\LC_{\lambda}$ is semi-simple.
\end{thm}

\begin{remark}
It is a deep result of de Cataldo and Migliorini \cite{dCM} that in fact, these
intersection forms are definite (and hence non-degenerate)
over $\QM$. The semi-simplicity of
each $\LC_{\lambda}$ over a field of characteristic zero
is much more straightforward (and is also pointed out in \cite{dCM}):
For relevant strata $X_{\lambda}$
the stalks of $\LC_{\lambda}$ 
have a basis at any point $x \in X_{\lambda}$ consisting of the irreducible components of
maximal dimension in the fibre $f^{-1}(s)$. The monodromy action
permutes these components. Hence each local system factors through a
representation of a finite group, and hence is semi-simple.

Remember that here, we do not assume condition \eqref{assump-parity}. Otherwise,
the semi-simplicity of the local systems would be automatic.
\end{remark}

Assume that $X_{\lambda}$ is a closed stratum and set 
$\FC := f_* \uk_{\widetilde{X}}[d_{\widetilde{X}}]$.
For any $\LC \in \Loc(X_{\lambda})$ we are
interested in the pairing 
\[ \Hom(i_{\lambda *} \LC[d_{\lambda}], \FC )
\times \Hom(\FC, i_{\lambda*}
\LC[d_{\lambda}]) \to \End(i_{\lambda*}\LC[d_{\lambda}]). \]
Applying two adjunctions on each side, this is equivalent to
determining the pairing
\[
\Hom(\LC, \LC_{\lambda}) \times
\Hom(  \LC_{\lambda}^{\vee}, \LC)
\to \End(\LC) \]
where, given  morphisms
\begin{gather*}
\LC \stackrel{f}{\to} \LC_{\lambda} \text{ and }
\LC_{\lambda}^{\vee} \stackrel{g}{\to} \LC
\end{gather*}
their pairing is given by the composition
\[
\LC \stackrel{f}{\to}
\LC_{\lambda} \to i_{\lambda}^! \FC [- d_\lambda] \to 
i_{\lambda}^* \FC [- d_\lambda]\to  \LC_{\lambda}^{\vee}
\stackrel{g}{\to} \LC
\]
where all morphisms except $f$ and $g$ are canonical. Hence it is
important to understand the morphism
\begin{equation}
\label{eq:intformmorph}
D_{\lambda} :  \LC_{\lambda} \to \LC_{\lambda}^{\vee} . \end{equation}

In the following lemma (and its proof), we use the notations in the diagram \eqref{diag-gccart}.

\begin{lem}\label{lem:schnittform}
Given $ x \in X_{\lambda}$ as above, the 
stalk of $D_{\lambda}$ at $x$ may be canonically identified 
with the morphism
\[
\HBM_{d_X - d_{\lambda}}(F_{\lambda}) \to 
\HBM_{d_X - d_{\lambda}}(F_{\lambda})^*
\]
induced by the intersection form associated to $X_{\lambda}$. \end{lem}

\begin{proof} Without loss of generality we may assume that $X_{\lambda} =
  U$, $X = N_{\lambda} \times U$ and $\widetilde{X} =
  \widetilde{N_{\lambda}} \times U$, for some contractible open
  subset $U \subset \CM^{d_{\lambda}}$. It follows that the stalk of
  $D_{\lambda}$ may be identified with the morphism
\[
\HC^0 (i^! \pi_*\uk_{\widetilde{N_{\lambda}}}[d_{\widetilde{X}}-d_{\lambda}]) \to 
\HC^0 (i^* \pi_* \uk_{\widetilde{N_{\lambda}}} [d_{\widetilde{X}} - d_{\lambda}])
\]
in which case the result follows from Corollary \ref{cor-inducedmorph}.
\end{proof}

Applying the adjunction 
$(-\otimes \LC_{\lambda}, - \otimes \LC_{\lambda}^{\vee})$ to
$D_{\lambda}$ we obtain a morphism
\[
B_{\lambda} : \LC_{\lambda} \otimes \LC_{\lambda} \to \uk_{\lambda}
\]
and it follows from the above lemma that the stalk of this 
morphism at each point $x \in X_{\lambda}$ is given by the intersection form
on $\HBM_{d_{X} - d_{\lambda}}(\pi^{-1}(x))$.

Let $j$ denote the open inclusion of the complement of $X_\lambda$. We
are now in a position to prove:

\begin{prop} \label{prop:extend} We have that 
$i_{{\lambda}*}\LC_{\lambda}[d_\lambda]$ is a direct summand of $f_* \uk_{\widetilde{X}}[d_{\widetilde{X}}]$
  if and only if the intersection form associated to $X_{\lambda}$ is
  non-degenerate. If this is the case we have an isomorphism
\begin{equation*}
f_* \uk_{\widetilde{X}}[d_{\widetilde{X}}]
\simeq
i_{\lambda *}\LC_{\lambda}[d_\lambda] \oplus j_{!*} j^* f_* \uk_{\widetilde{X}}[d_X].
\end{equation*}
\end{prop}

Note that Theorem \ref{thm-intformsgeneral} now follows by a simple
induction over the stratification.

\begin{proof} The above discussion shows that
$i_{\lambda*}\LC_{\lambda}[d_\lambda]$
is a direct summand of
$f_* \uk_{\widetilde{X}}[d_{\widetilde{X}}] \Iff D_{\lambda}$
is an isomorphism $\Iff B_{\lambda}$ is non-degenerate $\Iff$ 
the intersection form associated to $X_{\lambda}$ is non-degenerate.

Now assume that $i_{\lambda*} \LC_{\lambda}[d_\lambda]$
is a direct summand of $f_* \uk_{\widetilde{X}}[d_X]$
and write
$f_* \uk_{\widetilde{X}}[d_X] \simeq i_{\lambda *}\LC_{\lambda}[d_\lambda] \oplus \FC$
for some perverse sheaf $\FC$. Then $\FC$ is necessarily self-dual
because $f_* \uk_{\widetilde{X}}[d_X]$ and $i_{\lambda*}\LC_{\lambda}[d_\lambda]$ are. Also
$\mathcal{H}^{m}(i_{\lambda}^* \FC) = 0$ for $m \ge
-d_{\lambda}$. Hence $\FC \cong j_{!*} j^* f_*
\uk_{\widetilde{X}}[d_X] $ by the characterisation of $j_{!*}$ given
in \cite[Proposition 2.1.9]{BBD}. \end{proof}

\subsection{Decomposing parity sheaves}
\label{subsec-decompparity}

In this section we keep the notation from the previous section, but
remove the semi-small assumption on $f$ and assume instead that our
stratified variety $X$ satisfies \eqref{assump-parity} and that $f$ is
even. Recall that the parity assumption implies that, if it exists,
the parity sheaf $\EC(\lambda, \LC)$ corresponding to an irreducible
local system $\LC \in \Loc(X_{\lambda})$ is well-defined up to
isomorphism.

Let $f$ be a proper, surjective, even morphism:
\[
f : \widetilde{X} \to X.
\]
 It follows that 
$f_*\uk_{\widetilde{X}}[d_{\widetilde{X}}]$ is a parity complex, and
hence may be decomposed into a direct 
sum of parity sheaves
\[
f_*\uk_{\widetilde{X}}[d_{\widetilde{X}}] \cong \bigoplus \EC(\lambda, \LC)[-n] ^{\oplus
m_n(\lambda, \LC)}.
\]
In this section we consider the problem of determining $m_n(\lambda, \LC) \in \NM$.

\begin{remark}
\label{rem-existence}
While we do not assume that $\EC(\lambda,\LC)$ exists for all pairs
$(\lambda,\LC)$, all the indecomposable summands of the
direct image are parity sheaves and thus existence will
follow for any pair occuring with non-zero multiplicity. Moreover, if
$f$ is semi-small, the multiplicities $m_n(\lambda, \LC) =0$ for $n
\ne 0$ and any $\EC(\lambda,\LC)$ which occurs in the direct image is
perverse.
\end{remark}

Define\footnote{We use the notation $\LC^\bullet_\l$ to emphasise the fact that
$\LC^\bullet_\l$ is not necessarily concentrated in one degree, as was
the case in the previous section.}
\[
\LC^\bullet_\l := i_\l^! f_* \uk_{\widetilde{X}}[d_{\widetilde{X}}].
\]
Because $f$ is even, if we define
\[
\LC_\l^n := \HC^n(\LC_\l^\bullet)
\]
then $\LC_\l^n$ is zero if $n \not \equiv d_{\widetilde{X}} (\textrm{mod
  2})$, and we have an isomorphism (which we fix)
\begin{equation} \label{eq:decomp1}
\LC_\l^\bullet \cong \bigoplus \LC_\l^n [-n].
\end{equation}
Note that $i_\l^*f_*\uk_{\widetilde{X}}[d_{\widetilde{X}}] \cong \DM
i_\l^!f_*\uk_{\widetilde{X}}[d_{\widetilde{X}}]$ because $f$ is proper
and $\widetilde{X}$ is smooth (and so
$\uk_{\widetilde{X}}[d_{\widetilde{X}}]$ is self-dual). Hence
\begin{equation} \label{eq:decomp2}
i_\l^*f_* \uk_{\widetilde{X}}[d_{\widetilde{X}}] \cong \DM\LC^\bullet
\cong \bigoplus (\LC_\l^n)^{\vee} [2d_\l + n].
\end{equation}
As in the previous section we are interested in the canonical morphism
\[
\LC^\bullet_\l = i^!f_* \uk_{\widetilde{X}}[d_{\widetilde{X}}] \stackrel{D_\l}{\longrightarrow}
i^* f_* \uk_{\widetilde{X}}[d_{\widetilde{X}}] \cong \DM \LC_{\lambda}^\bullet.
\]
Taking cohomology and using the decompositions
\eqref{eq:decomp1} and \eqref{eq:decomp2} we get maps
\[
D_\lambda^n := \HC^{n}(D_\lambda) : \LC_\l^n \to (\LC_\l^{-2d_\l-n})^{\vee}.
\]

\begin{remark} One may interpret the maps $D_\l^n$ in
  terms of an intersection form as follows. Fix a point $x
  \in X_\l$, and let $F_\l$ denote the fibre of $f$ over $x \in
  X$. Then, one may identify
  the stalk of $\LC_\l^n$ (resp. of $(\LC_\l^{-2d_\l-n})^{\vee}$) at $x$ with 
$\HBM_{d_{\widetilde{X}} - n - 2d_\l} (F_\l)$
(resp. $\HBM_{d_{\widetilde{X}}  + n} (F_\l)^*$). Then, using similar
arguments to the proof of Lemma \ref{lem:schnittform} one may
identify $D_\l^n$ with the map
\[
\HBM_{d_{\widetilde{X}} - n - 2d_\l} (F_\l) \to \HBM_{d_{\widetilde{X}}  + n} (F_\l)^*
\]
induced by the intersection form $
\HBM_{d_{\widetilde{X}} - n - 2d_\l} (F_\l) \times
\HBM_{d_{\widetilde{X}}  + n} (F_\l) \to k.$

The following theorem shows that knowledge of these intersection
forms allows one to decompose $f_* \uk_{\widetilde{X}} [
d_{\widetilde{X}}]$ into parity sheaves.
\end{remark}

\begin{thm}
\label{thm-paritydecomp}
We have an isomorphism
\[
f_* \uk_{\widetilde{X}}[d_{\widetilde{X}}] \cong \bigoplus_{\l \in \L;
  n \in \ZM} \EC(\lambda, \LC^n_{\lambda}/ \ker D^n_{\lambda})[-n-d_\l]. \]
In particular, the multiplicity $m_n(\lambda, \LC)$ of an indecomposable 
parity sheaf $\EC(\lambda, \LC)[-n]$ as a direct summand of $f_*
\uk_{\widetilde{X}}[d_{\widetilde{X}}]$ is equal to the multiplicity of
$\LC$ in $\LC^{n-d_\l }_{\lambda}/ \ker D^{n-d_\l}_{\lambda}$.
\end{thm}

\begin{remark}
  If $f$ is semi-small then $\LC_\l^n$ is zero for $n > -d_\l$ and
  hence $D_\l^n = 0$ unless $n = -d_\l$. In this case 
  Theorem \ref{thm-paritydecomp} gives an isomorphism
\[
f_* \uk_{\widetilde{X}}[d_{\widetilde{X}}] \cong \bigoplus_{\l \in \L}
\EC(\lambda, \LC^{-d_\l}_{\lambda}/ \ker D^{-d_\l}_{\lambda}). \]
Assumption \eqref{assump-parity} guarantees that each local system
$\LC^{-d_\l}_{\lambda}/ \ker D^{-d_\l}_{\lambda}$ is
semi-simple. Hence the Decomposition Theorem is true
if and only if each $D_\l^{-d_\l}$ is an isomorphism, which is the
case if and only if each intersection form is
non-degenerate (by Lemma \ref{lem:schnittform}). We have already seen this result in Theorem
\ref{thm-intformsgeneral} in the context of an arbitrary semi-small
map. The advantage of the above theorem is that it explains (in the
restricted context of varieties satisfying the parity assumptions) how to
decompose $f_* \uk_{\widetilde{X}}[d_{\widetilde{X}}]$ when the
Decomposition Theorem fails.
\end{remark}

\begin{proof} By Proposition \ref{prop-restriction} it is enough to
  prove that, if $i_\l : X_{\lambda} \into X$ is the inclusion of a
  closed stratum, then
\[
f_* \uk_{\widetilde{X}}[d_{\widetilde{X}}] \cong \bigoplus ( i_{\l*} (
\LC^n_{\lambda}/ \ker D^n_{\lambda})[-n]) \oplus \GC
\]
where $\GC$ is a parity complex having no direct summands supported on
$X_\l$. However this is an immediate consequence of Proposition
\ref{prop:decompwithsupport}.
\end{proof}

\begin{remark} \label{rem:Odecomp}
In this section we have only considered the case of field
coefficients. However using \eqref{eq:mod hom parity} and idempotent
lifting (e.g. \cite[Theorem 12.3]{Feit}) one can show that if $\OM$ is
a complete discrete valuation ring with residue field $\FM$ and $\LC
\in \Loc(X_{\lambda}, \OM)$ (see Section ~\ref{subsec-naa}), then the graded multiplicity of
$\EC(\lambda,\LC)$ in 
$f_*\underline{\OM}_{\widetilde{X}}[d_{\widetilde{X}}]$ is equal to
the graded multiplicity of $\EC(\lambda,\FM \LC)$ in
$f_*\underline{\FM}_{\widetilde{X}}[d_{\widetilde{X}}]$. Hence the
results of this section also yield multiplicities with coefficients in
$\OM$.
\end{remark}

\section{Applications}
\label{sec-examples}

\subsection{(Kac-Moody) Flag varieties}
\label{subsec-KacMoody}

In this section we show the existence and uniqueness of parity sheaves
on Kac-Moody flag varieties. Throughout we only work with the trivial pariversity $\dagger = \natural$. 
The reader unfamiliar with Kac-Moody
flag varieties may keep the important case of a (finite) 
flag variety in mind. The standard reference for Kac-Moody Schubert
varieties is \cite{Ku}.

We begin by fixing some notation, which is identical to that of \cite{Ku}.
Let $A$ be a generalised Cartan matrix
of size $l$ and let $\gg(A)$ denote the corresponding 
Kac-Moody Lie algebra with Weyl group $W$, Bruhat order $\le$, length
function $\ell$ and 
simple reflections $S = \{ s_i\}_{i=1, \dots l}$.
To $A$ one may also associate a 
Kac-Moody group $\GC$ and subgroups $N$, $\BC$ and $\TC$ 
with  $\BC \supset \TC \subset N$.
Given any subset
$I \subset \{1, \dots, l \}$ one has a standard parabolic subgroup $\PC_I$
containing $\BC$ and a canonical Levi subgroup $\GC_I \subset \PC_I$.
The group $\TC$ is a connected
algebraic torus, $\BC$, $N$, $\PC_I$, $\GC_I$ and $\GC$ are all
pro-algebraic groups and $(\GC, \BC, N, S)$ is a Tits system with Weyl
group canonically isomorphic to $W$. The set $\GC / \PC_I$ may be given the structure of an 
ind-variety and is called a Kac-Moody flag variety.

Let $\hg_{\ZM}$ denote the lattice of cocharacters of $\TC$. Its dual
$\hg_{\ZM}^*$ may be identified with the lattice of characters of
$\TC$. In $\hg_{\ZM}^*$ one has the set $\Delta$ of roots, together with
a decomposition $\Delta = \Delta^+ \sqcup \Delta^-$ into the subsets of
positive and negative roots. Let $\Delta_{\re}$ denote the real
roots and $\Delta_{\re}^+ = \Delta_{\re} \cap \Delta^+$ and
$\Delta_{\re}^- = \Delta_{\re} \cap \Delta^-$ denote the positive and
negative real roots respectively. Finally, given a subset $I \subset
\{ 1, \dots, l\}$ we have have subsets $\Delta_I$, $\Delta_{I,\re}$,
$\Delta_{I,\re}^+$ etc. consisting of those (positive, real) roots in
the span of simple roots indexed by $I$.

\begin{ex}
If $A$ is a Cartan matrix then $\gg(A)$ is a 
semi-simple finite-dimensional complex Lie algebra
and  $\GC$ is the semi-simple 
and simply connected complex linear algebraic group with
Lie algebra $\gg(A)$, $\BC$ is a Borel subgroup,
$\TC \subset \BC$ is a maximal torus, $N$ is the normaliser
of $\TC$ in $\GC$,  $\PC_I$ is a standard parabolic 
and $\GC/ \PC_I$ is a partial flag variety.
\end{ex}

\begin{ex}
\label{ex-Gr}
If $A$ is now a Cartan matrix of size $l-1$ and $\gg(A)$ is 
the corresponding Lie
algebra with semisimple simply-connected group $G$, one can obtain
a generalised Cartan matrix $\tilde{A}$ by 
adding an $l$-th row and column with the values:
\[a_{l,l}=2 , a_{l,j}=-\alpha_j(\theta^\vee),
a_{j,l}=-\theta(\alpha_j^\vee),\]
where $1\leq j\leq l-1$, $\alpha_i$ are the simple roots of
$\gg(A)$ and $\theta$ is the highest root. The corresponding Kac-Moody
Lie algebra $\gg(\tilde{A})$ (resp. group $\GC$) is the so-called
(untwisted) affine Kac-Moody Lie algebra (resp. group) defined
in~\cite[Chapter 13]{Ku}. 
It turns out that the associated Kac-Moody flag
varieties have an alternative description as partial affine
flag varieties. Let $\KC=\CM((t))$
denote the field of Laurent series and $\OC=\CM[[t]]$ the ring of
Taylor series. Then, for example, the sets $G(\KC)/\IC$ (the affine flag
variety) and
$G(\KC)/G(\OC)$ (the affine Grassmannian) may be given the structure of an ind-variety and are
isomorphic to the Kac-Moody flag variety $\GC/\PC_I$ for $I =
\emptyset$ and $I=\{1,\ldots,l-1\}$ respectively. Here $G(\KC)$
 (resp. $G(\OC)$) denotes the group of $\KC$ (resp. $\OC$)-points of $G$
and $\IC$ denotes
the Iwahori subgroup, defined as the inverse image of a Borel
subgroup $B \subset G$ under the evaluation $G(\OC) \to G$.
\end{ex}

Given a subset $I \subset \{1, \dots, l \}$
we denote by $A_I$ the submatrix of $A$ consisting 
of those rows and columns indexed by $I$. For any such $I$, $A_I$
is a generalised Cartan matrix. A subset 
$I \subset \{1, \dots, l \}$ 
is of {\bf finite type} if $A_I$ is a Cartan matrix.
Equivalently, the subgroup $W_I \subset W$ generated by the
simple reflections $s_i$ for $i \in I$ is finite. Below we will mostly
be concerned with subsets $I\subset \{1, \dots, l \}$ of finite 
type.\footnote{
Much of the theory that we develop below is also valid $\GC / \PC_J$ 
even when $J$  is not 
of finite type, but we will not make this explicit.}

For any two subsets $I, J \subset \{1, \dots, l \}$ of finite
type we define
\[
{}^IW^J := \{ w \in W \; | \; s_iw > w \text{ and }ws_j > w
\text{ for all }
i \in I, j \in J \}.
\]
The orbits of $\PC_I$ on $\GC/ \PC_J$ give rise to a Bruhat 
decomposition:
\[ \GC/ \PC_J = 
\bigsqcup_{w \in {}^IW^J} \PC_Iw\PC_J/\PC_J 
= \bigsqcup_{w \in {}^IW^J} {}^IX^J_w. \]
The Bruhat decomposition gives an 
algebraic stratification of $\GC/ \PC_J$.

If $I = \emptyset$ each ${}^IX^J_w$ is isomorphic
to an affine space of dimension $\ell(w)$. In general
the decomposition of ${}^IX^J_w$ into orbits under $\BC$ gives
a cell decomposition
\begin{equation} \label{eq-flagcelldecomp}
{}^IX^J_w = \bigsqcup_{x \in W_IwW_J \cap {}^{\emptyset}W^J}
\CM^{\ell(x)}. \end{equation}

In the following proposition we analyse the strata 
${}^IX^J_w$.

\begin{prop}
\label{prop-flagstrata}
Let $k$ be a ring.
\begin{enumerate}
\item The graded $k$-module $H^{\bullet}({}^IX^J_w, k)$ is torsion free and concentrated
in even degrees.
\item The same is true of $H^{\bullet}_{\PC_I}({}^IX^J_w, k)$
if all the torsion primes for $A_I$ are invertible in 
$k$.
\end{enumerate}
Moreover, any local system or $\PC_I$-equivariant local system on
${}^IX^J_w$ is constant.
\end{prop}

We begin with the following lemma:

\begin{lem}
\label{lem:pi1}
  For any two subsets $I, J \subset \{ 1, \dots, l\}$ of finite type and
  $w \in {}^IW^J$, the variety ${}^IX_w^J$ is simply connected.
\end{lem}

\begin{proof}
  For each subset $I \subset \{ 1, \dots, l \}$ of finite type, we can find
  a cocharacter $\l_I : \CM^\times \to \TC$ such that $\langle \l_I, \alpha
  \rangle = 0$ if $\a \in \Delta_{I,\re}$ and $\langle \l_I, \alpha \rangle
  > 0$ for all $\alpha \in \Delta_{\re}^+ \setminus \Delta_{I,\re}$. By working in
  suitable charts around each $\TC$-fixed point on ${}^IX^J_w$ one may
  show that, for all $x \in {}^I X^J_w$,  we have
\[
\lim_{\CM^\times \ni z \to 0} \l_I(z) \cdot x \in \GC_I w
    \PC_J / \PC_J.
\]
A similar argument shows that $\GC_I w\PC_J / \PC_J$ is fixed by $\l_I(\CM^\times) \subset
\TC$. It follows that $\GC_Iw\PC_J/\PC_J$ is a deformation retract of
${}^IX^J_w$.

Now $\GC_Iw\PC_J/\PC_J$ is isomorphic to a (finite)
partial flag variety for $\GC_I$. It is standard that
partial flag varieties are simply connected.\footnote{One possible proof: Every partial flag variety is isomorphic to the partial flag variety of a simply connected algebraic group. Now any homogeneous space with connected stabilisers for a simply connected group is simply connected, by the long exact sequence of homotopy groups for a fibration.} Hence
\[
\pi_1({}^IX^J_w) = \pi_1(\GC_Iw\PC_J/\PC_J) = \{ 1 \}
\]
as claimed.
\end{proof}

\begin{proof}
[Proof of Proposition \ref{prop-flagstrata}]
The first statement follows from the fact that
\eqref{eq-flagcelldecomp} provides an affine paving of 
${}^IX^J_w$. By Lemma \ref{lem:pi1} each  ${}^IX^J_w$ is simply
connected, and hence any local system on $^IX^J_w$ is constant.
Now if $H$ is the reductive part
of the stabiliser of a point in ${}^IX^J_w$ then $H$ is isomorphic
to a regular reductive subgroup of a semi-simple connected and 
simply connected algebraic group with Lie algebra $\gg(A_I)$. It
follows that any $\PC_I$-equivariant local system on ${}^IX^J_w$ is
constant. We also have
\[ H^{\bullet}_{\PC_I}({}^IX^J_w, \ZM)
\cong H^{\bullet}_{H}(\pt, \ZM). \]
By Theorem \ref{thm-BGtorsion} this has no $p$-torsion for $p$ not a
torsion prime for $A_I$ and the result follows. \end{proof}

For the rest of this section we fix a complete local principal ideal 
domain $k$.

Fix $I, J \subset \{1, \dots, l \}$ of finite type.
We consider the following situations:
\begin{gather} \label{KMFV-const}
\text{ $ X = \GC / \PC_J$, an ind-variety stratified by $\PC_I$-orbits;}
\\ \label{KMFV-equi}
\text{$X = \GC / \PC_J$, an ind-$\PC_I$-variety.} \end{gather}
 If we are in situation \eqref{KMFV-equi}, we make the following
 assumption:
\begin{equation} \label{assump-torsionprimes}
\begin{array}{c}
\text{for all subsets $K \subset \{ 1, \dots, l\}$ of finite type} \\
\text{the torsion primes of $A_K$ are invertible in $k$.}
\end{array}
\end{equation}
\begin{remark}
 Using the results of Section \ref{sec-torsion} the above
 assumption \eqref{assump-torsionprimes} may be easily read off the
 Cartan matrix $A$. For example, any complete local principal ideal
 domain $k$ in which 2, 3 and 5 are invertible will always satisfy the
 above assumption.
\end{remark}

In either case we let $D_{I}(\GC/ \PC_J) := D(X,k) = D(X)$ be 
as in Section \ref{subsec-naa}
(see also Section \ref{sec:ind}).
Proposition \ref{prop-flagstrata} shows
that the stratified ind-($\PC_I$)-variety $\GC/ \PC_Y$ satisfies
\eqref{assump-parity} and \eqref{assump-parity2}.
By Theorem \ref{indecs}, it
follows that there exists  up to isomorphism at most one parity sheaf
with support $\ov{{}^IX_w^J}$ for each $w\in {}^I W^J$.

The first aim of this section is to show:

\begin{thm}
\label{thm-flagparity} 
Suppose that we are in situation \eqref{KMFV-const} or \eqref{KMFV-equi}.
For each $w\in {}^I W^J$, there exists, up to isomorphism, one parity
sheaf $\EC(w) \in D_{I}(\GC/ \PC_Y)$ with support $\ov{{}^IX_w^J}$.
\end{thm}

\begin{remark}
  If $I = \emptyset$, then the $\PC_I$-orbits on $\GC/\PC_J$ are isomorphic to affine spaces. In this case one can use the results of Section \ref{subsec:fieldcontract} to deduce the existence of $\dagger$-parity sheaves $\EC^\dagger(w)$ for all $w \in {}^IW^J$ and any pariversity in the setting \eqref{KMFV-const}.
\end{remark}

Recall that, if we are in the situation \eqref{KMFV-equi} then,
given any three subsets $I, J, K \subset \{ 1, \dots, l \}$
of finite type there exists a bifunctor
\begin{align*} 
D_I(\GC /\PC_J) \times D_J(\GC / \PC_K) & \to D_I(\GC / \PC_K) \\
(\FC, \GC) & \mapsto \FC * \GC
\end{align*}
called convolution (see \cite{SpIH, MV}). It is defined 
using the convolution diagram (of topological spaces):
\[
\GC /\PC_J \times \GC / \PC_K
\stackrel{p}{\leftarrow} \GC \times \GC / \PC_K 
\stackrel{q}{\to} \GC \times_{\PC_J} \GC / \PC_K
\stackrel{m}{\to} \GC / \PC_K \]
where $p$ is the natural projection, $q$ is the quotient map and 
$m$ is the map induced by multiplication. One sets
\[ \FC * \GC := m_* \KC
\quad \text{where} \quad q^* \KC \cong p^* (\FC \boxtimes \GC). \]
For the existence of $\KC$ and how to make sense of
$\GC \times \GC / \PC_K$ algebraically, we refer the reader to
\cite[Sections 2.2 and 3.3]{Na}.

The second goal of this Section is to show:
\begin{thm}
\label{thm-flagconv}
Suppose that we are in situation \eqref{KMFV-equi}. Then
convolution preserves parity: 
if $\FC \in D_I(\GC /\PC_J)$ and $\GC \in D_J(\GC / \PC_K)$ are parity
complexes,
then so is $\FC * \GC\in D_I(\GC / \PC_K)$.
\end{thm}

\begin{remark}
\label{rem:bourbaki}
The case of finite flag varieties was considered in \cite{SpIH}. There
Springer gives a new proof, due to MacPherson and communicated to him
by Brylinski, of the fact that the characters of intersection
cohomology complexes on the flag variety give the Kazhdan-Lusztig
basis of the Hecke algebra. This uses parity considerations in an
essential way. See also \cite{Soe}.
\end{remark}

Before turning to the proofs we prove some properties about the
canonical quotient maps between Kac-Moody flag varieties and recall
the definition of (generalised) Bott-Samelson varieties. Unless we
state otherwise, in all statements below we assume that we are in either
situation \eqref{KMFV-const} or \eqref{KMFV-equi}.

If $J \subset K$ are subsets of $\{1, \dots, l \}$ the canonical
quotient map
\[
\pi^J_K : \GC/ \PC_{J} \to \GC/ \PC_{K}.
\]
is a morphism of ind-varieties.

\begin{prop}
\label{prop-quotparity}
If $K$ is of finite type then both 
 $(\pi^J_K)_*$ and $(\pi^J_k)^*$ preserve parity.
\end{prop}

\begin{proof} For the duration the proof we abbreviate $\pi := \pi^J_K$.
 Because a complex is parity if and only if it is parity
after applying the forgetful functor, it is clearly enough to deal with
  the non-equivariant case (i.e. that we are in situation
  \eqref{KMFV-const}).
Moreover, as the stratification of 
$\GC/ \PC_{K}$ by $\BC$-orbits refines the stratification by
$\PC_I$-orbits we may assume without loss of generality that 
$I = \emptyset$. It is known (see the discussion of \cite{Ku} around Proposition 7.1.5) that $\pi$ is a 
stratified proper morphism
between the stratified ind-varieties $\GC/ \PC_{J}$ and  
$\GC/ \PC_{K}$. Moreover, the same proposition shows that the
restriction of $\pi$ to a stratum in $\GC/ \PC_{K}$ is simply
a projection between affine spaces. If follows that $\pi$ is 
even and hence $\pi_*$ preserves parity complexes by 
Proposition \ref{prop-evenpush}.

We now prove that $\pi^*$ preserves parity complexes. So assume that $\FC$ is parity, or equivalently that $\FC$ and $\DM \FC$ are $*$-parity. Then it is enough to show that $\pi^* \FC$ and $\DM \pi^* \FC \cong \pi^! \DM \FC$ are $*$-parity. This is clear for $\pi^* \FC$. For $\pi^! \DM \FC$ note that our assumptions on $K$ guarantee that $\pi$ is a smooth morphism with fibres of some (complex) 
dimension $d$. Hence $\pi^! \cong \pi^*[2d]$ and so $\pi^! \DM \FC \cong \pi^* \DM \FC[2d]$ is also $*$-parity.
\end{proof}

Now, let $I_0 \subset J_1 \supset I_1 
\subset J_2 \supset \dots \subset  J_n \supset
I_n$ be finite type subsets of $\{1, \dots, l \}$.
For $1 \le i \le k \le n$ consider the spaces
\begin{gather*} \Bs(i, \dots, k) := 
\PC_{J_i} \times^{\PC_{I_i}} \PC_{J_{i+1}} \times^ {\PC_{I_{i+1}}} 
\dots \PC_{J_{k-2}} \times^{\PC_{I_{k-1}}} \PC_{J_{k}} /
\PC_{I_k}, \\
Y(i, \dots, k) := 
\GC \times^{\PC_{I_i}} \PC_{J_{i+1}} \times^ {\PC_{I_{i+1}}} 
\dots \PC_{J_{k-2}} \times^{\PC_{I_{k-1}}} \PC_{J_{k}} /
\PC_{I_k}
\end{gather*}
defined as the quotient of $\PC_{J_i} \times \PC_{J_{i+1}} \times \dots \times
\PC_{J_k}$ (resp. $\GC \times \PC_{J_{i+1}} \times \dots \times
\PC_{J_k}$) by $\PC_{I_i} \times \PC_{I_{i+1}} \times \dots \times \PC_{I_k}$
where $(q_i, \dots, q_{k})$ acts on $(p_i, \dots, p_k)$ by
\[
(p_iq_i^{-1}, q_ip_{i+1}q_{i+1}^{-1}, \dots, q_{k-1}p_k q_k^{-1}).
\]
Then $Y(i,\dots,k)$ is a projective ind-$\GC$-variety and
$BS(i,\dots,k)$ is a closed subvariety. The space $BS(i,\dots,k)$ 
is called a generalised Bott-Samelson variety. (When $I_i =
\emptyset$ and $|J_i|= 1$ for all $i$ then 
$BS(i,\dots, k)$ is
constructed in \cite[7.1.3]{Ku}. The construction of $BS(i,\dots,k)$
in general is discussed in \cite{GL}. The construction of
$Y(i,\dots,k)$ is similar).
We will denote points in these varieties
by $[p_i,\dots, p_k]$. For $i \le j \le k$ we have a morphism
of ind-varieties $f_j : Y(i,\dots, k) \to \GC / \PC_{I_j}: [p_i, \dots, p_k] \mapsto p_i\dots p_j\PC_{I_j}.$
The map
\begin{align*}
Y(i,\dots, k) & \to \GC / \PC_{I_i} \times 
\dots \times \GC / \PC_{I_k} \\
p & \mapsto (f_i(p), \dots, f_k(p)) 
\end{align*}
is a closed embedding with image
\begin{equation} \label{eq-altBSimage}
\{ 
(x_i, \dots, x_k) \in \GC / \PC_{I_i} \times
\dots \times \GC / \PC_{I_k}
\;|\;
\pi^{I_j}_{J_{j+1}}(x_{j}) = \pi^{I_{j+1}}_{J_{j+1}}(x_{j+1}),\;i \le j \le k-1. 
\}
\end{equation}
The image of $BS(i, \dots,k) \subset Y(i,\dots,k)$ is the sublocus of such $(x_i,\dots,x_k)$ with
$\pi^{I_i}_{J_{i}}(x_i) = \PC_{J_{i}}/\PC_{J_{i}} \in \GC/\PC_{J_{i}}$ (see
\cite[Section 7]{GL}).
It follows that we have a diagram in which all squares are Cartesian 
(note that $Y(i) = \GC/\PC_{I_i}$):
\[
\small{
\xymatrix@=0.35cm{
\Bs(1,\dots,n) \ar[r] \ar[d] &  Y(1,2, \dots, n) \ar[r] \ar[d] & Y(2, \dots, n) \ar[d]
\ar[r] & \dots \ar[r] &
Y({n-1}, n) \ar[r] \ar[d] & \GC / \PC_{I_{n}} \ar[d]^{\pi^{I_n}_{J_n}}
\\
\Bs(1,\dots,n-1) \ar[r] \ar[d]&  Y(1,\dots, {n-1}) \ar[d] \ar[r] & Y(2, \dots, {n-1}) \ar[d] \ar[r]
 & \dots \ar[r] & \GC / \PC_{I_{n-1}}
\ar[r]^{\pi^{I_{n-1}}_{J_n}} \ar[d] & \GC/ \PC_{J_{n}}  \\
\vdots \ar[d] & \vdots \ar[d] & \vdots \ar[d] & \ddots \ar[d] \ar[r] &  \GC/ \PC_{J_{n-1}} \\
\Bs(1,2) \ar[r] \ar[d]  & Y(1,2) \ar[r]\ar[d] & \GC / \PC_{I_2} \ar[r]\ar[d] & \GC / \PC_{J_3} \\
\Bs(1) \ar[r] \ar[d] & \GC / \PC_{I_1}\ar[r]^{\pi^{I_1}_{J_2}} \ar[d]^{\pi^{I_1}_{J_1}}
 & \GC / \PC_{J_2} \\
\PC_{J_1} / \PC_{J_1}  \ar[r] & \GC / \PC_{J_1}
}
} \]
Let $f : \Bs(1,\dots,n) \to \GC/\PC_{I_n}$ denote the restriction of
$f_n$ to $\Bs(1,\dots,n)$; it agrees with the map along the top of the
above diagram.

\begin{prop} \label{prop:BSparity}
The sheaf
$f_{*} \uk_{\Bs(1, \dots, n)} \in D_{I_0}(\GC/ \PC_{I_n})$ is parity. \end{prop}

\begin{proof} (See \cite{Soe}.)
Repeated use of proper base change applied to the above diagram
gives an isomorphism
\[
f_{*} \uk_{\Bs(1, \dots, n)}
\cong (\pi^{I_n}_{J_{n}})^*
(\pi^{I_{n-1}}_{J_{n}})_*
\dots 
(\pi^{I_2}_{J_2})^* (\pi^{I_1}_{J_2})_*
(\pi^{I_1}_{J_1})^* \uk_{\PC_{J_1}}
\]
where $\uk_{\PC_{J_1}}$ denotes the skyscraper sheaf on the point
$\PC_{J_1}/\PC_{J_1} \in \GC / \PC_{J_1}$. However $\uk_{\PC_{J_1}}$ is
certainly parity and the result follows from Proposition
 \ref{prop-quotparity}.
\end{proof}

We can now prove Theorems \ref{thm-flagparity} and \ref{thm-flagconv}:

\begin{proof}
Fix subsets $I, J \subset \{ 1, \dots, l \}$ of finite type
and choose $w \in {}^I W^J$. By Theorem \ref{indecs} it is enough to
show that there exists at least one parity sheaf  $\EC$ such that 
the support of $\EC$ is $ \ov{ {}^IX_w^J} $.

One may show (see \cite[Proposition 1.3.4]{W-thesis}) that
there exists a sequence $I = I_0 \subset J_1 \supset I_1 
\subset J_2 \supset \dots J_n \supset I_n = J$ such that,
if $\Bs$ denotes the corresponding generalised Bott-Samelson
variety, the morphism
\[ f : \Bs \to \GC / \PC_J \]
has image $\ov{{}^IX_w^J}$ and is an isomorphism over ${}^I
X_w^J$.\footnote{
Actually the condition that $f : \Bs \to \GC / \PC_J$ be an
isomorphism over ${}^IX_w^J$ is not necessary for the proof. One only
needs that there exists a sequence $I = I_0 \subset J_1 \supset I_1 
\subset J_2 \supset \dots J_n \supset I_n = J$ such that the image
of the corresponding generalised Bott-Samelson variety in $\GC/\PC_J$
is equal to $\ov{{}^IX_w^J}$. In this case any indecomposable summand
of $f_* \uk_{\Bs}$ with support equal to $\ov{{}^IX_w^J}$ will give the
desired parity sheaf (up to a shift).
}
Let $d_{\Bs}$ denote the complex dimension of $\Bs$. Then $f_*
\uk_{\Bs}[d_{\Bs}]$ is self-dual (because $f$ is proper and $\Bs$ is smooth)
and parity (by Proposition \ref{prop:BSparity}). Hence if we let
$\EC$ denote the unique indecomposable direct summand of $f_*
\uk_{\Bs}[d_{\Bs}]$ which is non-zero over ${}^I X_w^J$ then $\EC$ is a
parity sheaf with support $\ov{ {}^IX_w^J}$. Theorem 
\ref{thm-flagparity} then follows in either situation
\eqref{KMFV-const} and \eqref{KMFV-equi}.

We now turn to Theorem \ref{thm-flagconv} and assume we are in the
situation \eqref{KMFV-equi}.
 By the 
uniqueness of parity sheaves, and the above remarks, it is
enough to show that if
\begin{gather*}  I = I_0 \subset J_1 \supset I_1 
\subset \dots \subset J_n \supset I_n = J \\
J = I_n \subset J_{n+1} \supset I_{n+1} \subset 
\dots \subset J_m \supset I_m = K
\end{gather*}
are two sequences of finite type subsets of $\{1, \dots, l \}$,
$\Bs_1$ and $\Bs_2$ are the corresponding generalised
Bott-Samelson varieties and $f_1 : \Bs_1 \to \GC/\PC_{J}$ and 
$f_2 : \Bs_2 \to \GC/\PC_{K}$ then
\[ f_{1*} \uk_{\Bs_1} * f_{2*} \uk_{\Bs_2} \in D_I(\GC/\PC_K) \]
is parity.

However, if $\Bs$ denotes the Bott-Samelson variety associated
to the concatenation $I = I_0 \subset J_1 \supset \dots \supset I_n
\subset \dots \subset J_m \supset I_m = K$ and $f : \Bs \to
\GC/\PC_K$ is the multiplication morphism then
\[ f_{1*} \uk_{\Bs_1} * f_{2*} \uk_{\Bs_2} 
\cong f_* \uk_{\Bs} \]
and the result follows from the proposition above.
\end{proof}

\begin{remark}~
\begin{enumerate}
\item Such theorems have been established for the finite flag
varieties if $k$ is a field of characteristic larger than the
Coxeter number by Soergel in \cite{Soe}. 
\item 
An important special case of the above is the affine Grassmannian. In
this case, parity sheaves are closely related to tilting modules
\cite{JMW3}.
\end{enumerate}
\end{remark}

\subsection{Toric varieties}
\label{subsec-toric}
In this section we prove the existence and uniqueness of $\natural$-parity sheaves on toric varieties. As in the previous section, here parity sheaf means $\natural$-parity sheaf.

For notation, terminology, and basic properties of toric varieties we
refer the reader to \cite{Ful} and \cite{CLS}. In this 
section $T$ denotes a connected algebraic torus and $M = X^*(T)$ 
and $N= X_*(T)$ denote the character and cocharacter lattices
respectively. If $L$ is a lattice 
we set $L_{\QM} := L \otimes_{\ZM} \QM$.

Recall that a fan in $N$ is a collection $\Delta$ of
polyhedral, convex cones in $N_{\QM}$ closed under
taking faces and intersections. To a fan $\Delta$ in $N$
one may associate a toric variety $X(\Delta)$ which is a connected
normal $T$-variety.  We write $X(\Delta,N)$ to specify the lattice if it is not clear from context.

There are finitely many orbits of $T$ on 
$X(\Delta)$ and the decomposition into orbits gives a stratification
\[ X(\Delta) = \bigsqcup_{\tau \in \Delta} O_{\tau} \]
indexed by the cones of $\Delta$. For example the zero cone $\{0\}$
always belongs to $\Delta$ and $O_{\{0\}}$ is an open dense orbit, canonically
identified with $T$.

In this section we fix a ring of coefficients $k$ as in Section
\ref{subsec-naa}, take
\begin{equation}
\text{$X = X(\Delta)$ as a $T$-variety}
\end{equation}
and let $D_{T}(X(\Delta)) = D(X)$ be as in Section
\ref{subsec-naa}. We use the notation of Section \ref{sec-dafp}
without further comment.

\begin{thm}
\label{thm-toricparity}
For each orbit $O_\tau$, there exists up to isomorphism one 
parity sheaf $\EC(\tau) \in D_{T}(X(\Delta))$
with support $V(\tau) = \ov O_{\tau}$.
\end{thm}

Let $\tau \in \Delta$ and let $N_{\tau}$ denote the
intersection of $N$ with the linear span of $\tau$. Then
$N_{\tau}$ determines a connected subtorus $T_{\tau} \subset T$. 

\begin{lem} 
The stabiliser of a point $x \in O_{\tau}$ is $T_{\tau}$ and 
is therefore connected.
\end{lem}

\begin{proof} This follows from the last exercise of Section 3.1 in 
\cite{Ful}. \end{proof}

We now turn to the proof of the theorem.

\begin{proof}
By the quotient equivalence, the categories of $T$-equivariant 
local systems on $O_{\tau}$ and $T_{\tau}$-equivariant local systems
on a point are equivalent. Hence any torsion free 
equivariant local system on $O_{\tau}$ is 
isomorphic to a direct sum of copies of the trivial local system $\uk_{\tau}$.
We have
\[
\Hom^{\bullet}(\uk_{\tau}, \uk_{\tau}) = 
H^{\bullet}_T(O_{\tau}) = H^{\bullet}_{T_{\tau}}(\pt) \]
which is torsion free and vanishes in odd degrees. It follows that
the $T$-variety $X(\Delta)$ satisfies
\eqref{assump-parity} and \eqref{assump-parity2}. By Theorem
\ref{indecs}, we conclude that 
for each $\tau \in \Delta$ there exists at most one
parity sheaf $\EC(\tau)$ supported on $V(\tau)$ and
satisfying $i_{\tau}^*\EC(\tau) \cong \uk_{\tau}[d_{\tau}]$.

It remains to show existence. Recall the following properties 
of toric varieties:
\begin{enumerate}
\item \label{it:toric1}
For $\tau \in \Delta$, $V(\tau)$ is a toric variety
for $T/T_{\tau}$ (\cite[Section 3.1]{Ful}).
\item \label{it:toric2}
For any fan $\Delta$ there exists a refinement $\Delta^{\prime}$
of $\Delta$ such that $X(\Delta^{\prime})$ is quasi-projective and 
the induced $T$-equivariant morphism
\[ \pi : X(\Delta^{\prime}) \to X(\Delta) \]
is a resolution of singularities (\cite[Section 2.6]{Ful}).
\item \label{it:toric3}
For all $\tau$ in $\Delta$ we have a Cartesian diagram (all
morphisms are $T$-equivariant):
\[
\xymatrix{ 
O_{\tau}\times Z \ar[r] \ar[d]^{\pi^{\prime}} \ar@/^.6cm/[rr]^{i_{\tau}^{\prime}}
& O_{\tau} \times X(\Sigma,N_\tau) \cong X(\Sigma,N)
\ar[r] \ar[d]  &  X(\Delta^{\prime}) \ar[d]^{\pi} \\
O_{\tau} \times \{\gamma_\tau\} \ar[r] \ar@/_.6cm/[rr]_{i_{\tau}}
& O_{\tau} \times U_{\tau,N_\tau} \cong U_{\tau,N} \ar[r] & X(\Delta) }
\]
Here the $U_{\tau,N}$ and $U_{\tau,N_\tau}$ denote the affine toric varieties for the cone $\tau$ in $N$ and $N_\tau$, while $\Sigma$ denotes the fan consisting of all cones in $\Delta^\prime$ contained in $\tau$. 
The square on the left is the product of $O_\tau$ with a fibre diagram.
\end{enumerate}

By \eqref{it:toric1} it suffices to show the existence of $\EC(\tau)$
when $\tau$ is the zero cone (corresponding to the open
$T$-orbit). For this it suffices to show that $\pi_*
\uk_{X(\Delta^{\prime})}$ is even. In fact, as
$\uk_{X(\Delta^{\prime})}[d_{\tau}]$ is self-dual and $\pi$ is proper,
we need only show that $\pi_*\uk_{X(\Delta^{\prime})}$ is $*$-even.

By proper base change we have $i_{\tau}^* \pi_*
\uk_{X(\Delta^{\prime})} 
\cong \pi_*^{\prime} \uk_{O_{\tau} \times Z}$. Under the quotient
equivalence $D_T(O_{\tau}) \stackrel{\sim}{\to} D_{T_{\tau}}(\pt)$,
the sheaf $\pi_*^{\prime} \uk_{O_{\tau} \times Z}$ corresponds to
$\tilde{\pi}_* \uk_Z \in D_{T_{\tau}}(\pt)$, where 
$\tilde{\pi}: Z \to \pt$ is the projection 
(of $T_{\tau}$-varieties). We will see in the proposition 
below that $\tilde{\pi}_* \uk_Z$ is 
always $*$-even. This proves the theorem.

\end{proof}

\begin{prop} Let $\tau \subset N_{\QM}$ be a full-dimensional polyhedral convex cone, $U_{\tau}$ the corresponding affine toric variety, and  
$\Delta^{\prime}$ a refinement of $\tau$ such that the corresponding
toric variety $X(\Delta^{\prime})$ is smooth and quasi-projective.
Let $x_{\tau}$ denote the unique
$T$-fixed point of $U_{\tau}$. Consider the Cartesian
diagram:
\[
\xymatrix{
Z = \pi^{-1}(x_{\tau}) \ar[d]^{\pi} \ar[r] & X(\Delta^{\prime}) \ar[d]^{\pi} \\
\{ x_{\tau} \} \ar[r] & U_{\tau} }
\]
Then $\pi_* \uk_{Z} \in D_T(\pt)$ is a direct sum of equivariant 
constant sheaves concentrated in even degree.
\end{prop}

\begin{proof} It is enough to show that the 
$T$-equivariant cohomology of $Z$ with integral coefficients is free over
$H^{\bullet}_T(\pt, \ZM)$ and concentrated in even degrees.
We will show that the integral cohomology of $Z$
is free, and generated by the classes of $T$-stable closed subvarieties.
The result then follows by the Leray-Hirsch lemma 
(see \cite[proof of Theorem 4]{BriToric}).

We claim in fact that $Z$ has a $T$-stable affine paving, which 
implies the result by the long exact sequence of 
compactly supported cohomology.
The argument is a straightforward adaption of \cite[10.3 -- 10.7]{Dani}
(which the reader may wish to consult for further details).

As $X(\Delta^{\prime})$ is assumed to be quasi-projective we can
find a piecewise linear function $g : N_{\QM} \to \QM$ which is strictly
convex with respect to $\Delta^{\prime}$. In other words, $g$ is
continuous, convex and for each maximal cone $\sigma \in \Delta^{\prime}$, 
$g$ is given on $\sigma$ by $m_{\sigma} \in M$. The function $g$
allows us to order the maximal cones of $\sigma$ as follows: We fix a 
generic point $x_0 \in N_{\QM}$ lying in a cone of  
$\Delta^{\prime}$ and declare that $\sigma^{\prime} > \sigma$ if
$m_{\sigma^{\prime}}(x_0) > m_{\sigma}(x_0)$. If $\sigma^{\prime}$ and
$\sigma$ satisfy $\sigma^{\prime} > \sigma$ and 
intersect in codimension 1, then their intersection is said to
be a positive wall of $\sigma$. Given a maximal cone $\sigma$
we define $\gamma(\sigma)$ to be the intersection of $\sigma$ with
all its positive walls.

It is then easy to check (remembering that $X(\Delta^{\prime})$ is
assumed smooth) that if we set
\[
C(\sigma) = \bigsqcup_{\gamma(\sigma) \subset \omega \subset \sigma}
O_{\omega}
\]
then $C(\sigma)$ is a locally closed subset of $X(\Delta^{\prime})$ 
isomorphic to an affine space of dimension equal to the codimension
of $\gamma(\sigma)$ in $N_{\QM}$. Lastly note that
\[
Z = \bigsqcup O_{\sigma}
\]
where the union takes place over those cones in $\Delta^{\prime}$ which are not
contained in any wall of $\tau$. Hence the order on maximal cones
yields a filtration of $Z$ by 
$T$-stable closed subspaces $\dots \subset F_{\sigma_{i+1}}
\subset F_{\sigma_{i}} \subset \dots $ such that $F_{\sigma_{i+1}}
\setminus  F_{\sigma_{i}}$ is isomorphic to an affine space for all
$i$. The result then follows. \end{proof}

\begin{remark}
With notation as above, $X(\Delta^{\prime})$ retracts equivariantly
onto $Z$. With this in mind, the above arguments (together with the
reduction to the quasi-projective case in \cite{Dani}) can be used to
establish the equivariant formality (over $\ZM$) of convex smooth
toric varieties.  The elegant Mayer-Vietoris spectral sequence
argument of \cite{BrZh} may then be used to identify the equivariant
cohomology ring with piecewise integral polynomials on the fan. This
is probably well-known to experts.
\end{remark}

\subsection{Nilpotent cones}
\label{subsec-Nilp}

Let $\NC$ denote the nilpotent cone in the Lie
algebra $\mathfrak{g}$ of a connected reductive group $G$. 
The group $G$ acts on $\NC$
by the adjoint action and has finitely many orbits~\cite{Rich}.
In this section we discuss the existence and uniqueness of
$\natural$-parity sheaves on $\NC$ stratified by the $G$-orbits,
considered as a $G$-variety.
Recall that the nilpotent orbits are even dimensional (e.g., \cite[\S 1.4]{CM}), so
$\natural = \Diamond$. All parity sheaves will be with respect to this pariversity.
For $x\in\NC$, let $A_G(x) = G_x/G_x^0$ and $C_x =
(G_x^0)^\red$ the maximal reductive quotient of $G_x^0$.
We fix a ring of coefficients $k$ as in Section \ref{subsec-naa} and
assume that for all $x\in \NC$, the torsion primes of $C_x$ (see Section \ref{sec-torsion}) and
the order of the group $A_G(x)$ are invertible in $k$.

\subsubsection{Uniqueness}

\begin{lem} \label{lem-uniq}
The parity conditions \eqref{assump-parity} and \eqref{assump-parity2} are satisfied.
\end{lem}

\begin{proof}
For any orbit $\OC \subset \NC$ and $x \in \OC$, let $\tilde \OC =
G/G_x^0$ and $\pi: \tilde \OC \to \OC$ be the finite Galois cover
given by $gG_x^0 \mapsto g \cdot x$ with Galois group $A_G(x)$.  
Note that $\Loc_{,G}(\OC)$ is equivalent to the category of $k[A_G(x)]$-modules that
are free over $k$.  Using the assumption that $|A_G(x)|$ is invertible
in $k$, one can show that any $k[A_G(x)]$-module free over $k$ is
projective (and hence a direct summand of a direct sum of copies of
the regular representation).  The regular representation corresponds
to the pushforward $\pi_* \uk_{\tilde \OC}$. It thus suffices to show
that the equivariant cohomology groups of $\pi_* \uk_{\tilde \OC}$ are
free $k$-modules and vanish in odd degrees. We have
\[
H^\bullet_G(\OC, \pi_* \uk_{\tilde \OC})
= H^\bullet_G(\tilde \OC) 
= H^\bullet_{G_x^0}(\pt)
= H^\bullet_{C_x}(\pt).
\] 
Using the assumption that torsion primes for $C_x$ are invertible, we apply
Theorem~\ref{thm-BGtorsion} to conclude that the left hand side is a
free $k$-module and vanishes in odd degrees.
\end{proof}

By Theorem \ref{indecs} we conclude that for each pair $(\OC,\LC)$
consisting of a nilpotent orbit together with an irreducible
$G$-equivariant local system, there is at most one parity sheaf
$\EC(\ov\OC,\LC)$ with support $\ov\OC$ extending $\LC[d_\OC]$.

\begin{remark}
Our restriction on the ring of coefficients can be reformulated in
terms of the root datum $(\XB,\Phi,\YB,\Phi^\vee)$ of $G$.
In~\cite{Herpel}, Herpel defines a
notion of pretty good prime: a prime $p$ is pretty good for $G$
if the groups $\XB/\ZM\Phi_1$ and
$\YB/\ZM\Phi_1^\vee$ have no $p$-torsion for all subsets
$\Phi_1\subset \Phi$. One has a chain of implications: very good
$\Longrightarrow$ pretty good $\Longrightarrow$ good.  The class of
reductive groups for which $p$ is pretty good is characterised by the
following properties (this is a variant of \cite[Remark 5.4]{Herpel}):
\begin{enumerate}
\item it contains all simple groups for which $p$ is very good;
\item it contains $GL_n$ for all $n$;
\item it is closed under taking products, replacing $G$ by a
  $p$-separably isogenous group, and replacing $G = H\times S$ by $H$
  if $S$ is a torus.
\end{enumerate}
Using the tables of centralisers from~\cite{Carter} and the above
characterisation, one can show that a prime $p$ is pretty good for $G$
if and only if for all $x\in \NC$, $p$ is not a torsion prime for
$C_x$ and does not divide the order of $A_G(x)$.
\end{remark}

\subsubsection{Existence}

It is known \cite[V, Theorem 24.8]{CS} that the intersection
cohomology complexes of nilpotent orbit closures, with coefficients in
any irreducible $G$-equivariant local system in characteristic zero,
are even. Thus a similar result holds for almost all
characteristics (see Proposition \ref{prop:almostall}). However, work still
needs to be done to determine precise bounds on $p$ for parity sheaves
to exist, resp. to be perverse, resp. to be intersection cohomology
sheaves.  In what follows, we begin to address these questions.

Springer's resolution $\pi : \widetilde \NC := G
  \times^B \ug \to \ov\OC_\reg = \NC$ is
semi-small and even \cite{DLP} (here $B$ is a fixed Borel subgroup of
$G$ with unipotent radical $U$ and $\ug = \Lie U$). Thus
$\EC(\ov\OC_\reg)$ exists and is perverse.  By Remark
\ref{rem-existence}, we also have existence of $\EC(\ov\OC,\LC)$ for
all pairs appearing with non-zero multiplicity in the direct image
$\pi_* \uk_{\tilde\NC} [\dim \NC]$. By semi-smallness, all of
these are perverse. We remark that if $|W|$ is invertible in $k$, then
those pairs are ``the same as in characteristic zero''.

In the case $G=GL_n$, every orbit $\OC$ is equivariantly
simply-connected and there is a natural $G$-equivariant semi-small
resolution of singularities of $\ov\OC$ whose fibres admit affine
pavings \cite{Brundan-Ostrik}. It follows that there exists a perverse
parity sheaf $\EC(\ov\OC)$ with support $\ov \OC$ for any $k$ as above.

For a nilpotent orbit $\OC$ in an arbitrary connected reductive group,
let us recall how to construct a ``standard'' resolution of $\ov\OC$
\cite{Pa}. Let $x$ be an element of $\OC\cap\ug$. By the
Jacobson-Morozov theorem, there is an $\sg\lg_2$-triple $(x,h,y)$ in
$\gg$. The semi-simple element $h$ induces a grading on $\gg$, and 
we can choose the triple so that all the simple root vectors have degree
$0$, $1$ or $2$. Let $P$ be the standard parabolic subgroup of $G$
corresponding to the set of simple roots with degree zero.
Then there is a resolution of the form $\pi_\OC : \widetilde \NC_\OC \to \ov\OC$, where
$\widetilde \NC_\OC = G \times^P \gg_{\geq 2}$ is a
$G$-equivariant subbundle of $T^* (G/P) = G \times^P \ug_P$
(here $\ug_P$ is the Lie algebra of the unipotent radical $U_P$ of $P$),
and $\pi_\OC$ is the restriction of the moment map. To settle the
question of existence for $\EC(\ov\OC,k)$ in general, one is lead to
the following problem.

\begin{question}
Is the resolution $\pi_\OC : \widetilde \NC_\OC \to \ov\OC$
even for any coefficients?
\end{question}

Given any parabolic subgroup $P$ of $G$
with Lie algebra $\pg$, and any $P$-stable ideal $\ig \subset \pg$,
consider the natural morphism $\pi_{P,\ig} : G\times^P\ig \to \gg$. Fresse recently
proved that if $G$ is of classical type then $\pi_{P,\ig}$ is even
\cite{Fresse}. This answers our question positively in this case,
taking $P$ as above and $\ig = \gg_{\geq 2}$.  In particular, it
follows that there exists a parity extension for a constant local
system on any nilpotent orbit of a classical group. In this way one
actually constructs parity sheaves for a possibly larger set of local systems,
but probably not more than those arising in the characteristic zero
Springer correspondence \cite[Conjecture 6.3]{Sommers}.

\subsubsection{Minimal singularities}
\label{subsec:min}

Suppose that $G$ is simple. Then there is a unique minimal
(non-trivial) nilpotent orbit in $\gg$. We denote it by $\OC_\mini$.
It is of dimension $d := 2 h^\vee -2$, where $h^\vee$ is the dual
Coxeter number \cite{WANG}. We conclude by studying the singularity
  $\ov\OC_\mini=\OC_\mini \cup \{0\}$.  
In this section we construct an indecomposable $G$-equivariant parity extension
of the constant sheaf $\uk[d]$ on $\OC_\mini$.  

Consider the resolution of singularities
\[
\pi : E := G \times^P \CM x_\mini \longto \ov\OC_\mini = \OC_\mini \cup \{0\}
\]
where $x_\mini$ is a highest weight vector of the adjoint
representation and $P$ is the parabolic subgroup of $G$ stabilising
the line $\CM x_\mini$. It is an isomorphism over $\OC_\mini$, and the
fibre above $0$ is the null section, isomorphic to $G/P$, which has
even cohomology. Hence $\pi$ is an even resolution, and so $\pi_*
\uk_E [d]$ is even.

\begin{remark} \label{remark-fudge}
The construction above works for any $k$ (in fact for any commutative ring).
However, the uniqueness theorem \ref{indecs} does not apply unless we
restrict to a $k$ for which \eqref{assump-parity} and
\eqref{assump-parity2} hold.  Rather than restrict to such $k$, we
work here in the more general setting where parity
sheaves may not be defined uniquely and so we can only discuss indecomposable
parity complexes.  One reason for doing this is that the singularities $\ov\OC_\mini$ arise
in the affine Grassmannian where the parity conditions are satisfied
for a larger class of coefficients.
\end{remark}

We begin with a general lemma for isolated singularities.

\begin{lem}
\label{lem:two strata}
Suppose $X = U \sqcup \{0\}$ is a stratified variety (thus
$0$ is the only singular point).
We denote by $j : U \to X$ and $i : \{0\} \to X$ the inclusions.

\begin{enumerate}
\item
\label{it:std filtr}
Let $\PC$ be a $*$-even complex on X whose restriction to $U$ is perverse.
Then we have a short exact sequence
\[
0 \longto \p j_! j^* \PC \longto \p H^0 \PC \longto i_* \p i^* \PC
\longto 0.
\]

\item
\label{it:leq -2}
If $\FC$ is any perverse sheaf on $X$ whose composition factors are
one copy of $\ic(X,\FM)$ and $N$ copies of $\ic(0,\FM)$, then
$\HC^m(\FC)_0 \simeq \HC^m(\ic(X,\FM))_0$ for all $m \leq -2$.
\end{enumerate}
\end{lem}

\begin{proof}
We have a distinguished triangle
\[
j_! j^* \PC \longto \PC \longto i_* i^* \PC \longtriright
\]
which gives rise to a long exact sequence of perverse cohomology
sheaves, which ends with:
\[
i_* \p H^{-1} i^* \PC \longto \p j_! j^* \PC \longto \p H^0 \PC
\longto i_* \p i^* \PC \longto 0.
\]
Now, $\p H^{-1} i^* \PC$ is identified with $(\HC^{-1}\PC)_0$ which is zero
since $\PC$ is $*$-even. This proves \eqref{it:std filtr}.

For \eqref{it:leq -2}, we proceed by induction on $N$. The result is
trivial for $N = 0$. Now suppose $N > 1$. There is a perverse sheaf
$\GC$ such that we have a short exact sequence of one of the two
following forms:
\begin{gather}
0 \longto \GC \longto \FC \longto \ic(0,\FM) \longto 0
\\
0 \longto \ic(0,\FM) \longto \FC \longto \GC \longto 0
\end{gather}
and we can consider the corresponding long exact sequence for the
cohomology of the stalk at zero.
From $\HC^m(\ic(0,\FM))_0 = 0$ for $m \leq -1$,
we deduce in both cases that $\HC^m(\FC)_0$ is isomorphic to
$\HC^m(\GC)_0$ for $m \leq -2$ (at least). The result follows by induction.
\end{proof}

\begin{prop}
\label{prop:min}
The following conditions are equivalent:
\begin{enumerate}
\item \label{min:parity perverse}
there exists a perverse, parity extension of $\un\FM_{\OC_\mini}[d]$;

\item \label{min:standard *-even}
the standard sheaf $\p j_! (\un\FM_{\OC_\mini} [d])$ is $*$-even;

\item \label{min:standard tf}
the standard sheaf $\p j_! (\un\OM_{\OC_\mini} [d])$ has torsion free stalks;

\item \label{min:H tf}
for all $m < d$, the cohomology group $H^m(\OC_\mini,\ZM)$ has no $p$-torsion;

\item \label{min:p}
the characteristic of $\FM$ is not one of the primes corresponding to the type of $G$ in the following table:
\[
\begin{array}{c|c|c|c|c}
A_n & B_n, C_n, D_n, F_4 & G_2 & E_6, E_7 & E_8\\
\hline
- & 2 & 3 & 2,3 & 2,3,5
\end{array}
\]
\end{enumerate}

\end{prop}

\begin{proof}
First suppose that there exists a parity complex $\EC$ extending $\un\FM_{\OC_\mini}$ that is also perverse. 
Then both $\EC$ and $\p j_! (\un\FM_{\OC_\mini} [d])$
are perverse sheaves whose composition factors are one copy of $\ic(\ov\OC_\mini,\FM)$ and some number of
copies of $\ic(0,\FM)$. 
By Lemma \ref{lem:two strata} \eqref{it:leq -2}, we have
\[
\HC^m(\p j_! (\un\FM_{\OC_\mini} [d]))_0 \simeq
\HC^m(\ic(\ov\OC_\mini,\FM))_0 \simeq
\HC^m(\EC)_0
\]
for $m \leq -2$.
Since $(\p j_! (\un\FM_{\OC_\mini} [d]))_0$ is concentrated in degrees
$\leq -2$, this proves that $\p j_! (\un\FM_{\OC_\mini} [d])$ is
$*$-even. Thus $\eqref{min:parity perverse} \Longrightarrow \eqref{min:standard *-even}$.

Now assume that $\p j_! (\un\FM_{\OC_\mini} [d])$ is $*$-even.
Consider the parity complex $\PC := \pi_* \un \FM_E[d]$ defined in the
discussion proceeding Remark~\ref{remark-fudge}.
By Lemma \ref{lem:two strata} \eqref{it:std filtr},
we have a short exact sequence
\[
0 \longto \p j_! (\un\FM_{\OC_\mini} [d]) \longto \p H^0 \PC
\longto i_* \p i^* \PC \longto 0.
\]
Since the extreme terms are $*$-even, we deduce that
$\p H^0 \PC$ is $*$-even as well. But $\p H^0 \PC$ is self-dual,
because $\PC$ is. Thus $\p H^0 \PC$ is parity.  The short exact
sequence also shows that it is an extension of $\un\FM_{\OC_\mini}$.
Thus $\eqref{min:standard *-even} \Longrightarrow \eqref{min:parity perverse}$. 

The equivalences $\eqref{min:standard tf} \Iff \eqref{min:H tf} \Iff
\eqref{min:p}$ are proved in \cite{cohmin,Ju-AffGr}.
Briefly, the stalk $\p \JC_!(\ov\OC_\mini,\ZM_p)_0$ is given by a
shift of $H^*(\OC_\mini,\ZM_p)$ truncated in degrees
$\leq d - 2$, and $H^{d - 1}(\OC_\mini,\ZM) = 0$, so
$\eqref{min:standard tf} \Iff \eqref{min:H tf}$.
Now, by a case-by-case calculation \cite{cohmin}, one finds that
$\eqref{min:H tf} \Iff \eqref{min:p}$.

The vanishing $H^{d - 1}(\OC_\mini,\OM) = 0$ implies that
$\p j_! (\un\FM_{\OC_\mini} [d]) = \FM \otimes^L_\OM \p j_! (\un\OM_{\OC_\mini} [d])$
by \cite{decperv}.
Thus $\eqref{min:standard *-even} \Iff \eqref{min:standard tf}$
by Proposition \ref{prop-reduction}.
\end{proof}

Finally, let us recall from \cite{decperv} when the standard sheaf is
equal to the intersection cohomology sheaf for a minimal singularity.

\begin{prop}
\label{prop:min std = ic}
Let $\Phi$ denote the root system of $G$, with some choice of positive
roots. Let $\Phi'$ denote the root subsystem of $\Phi$
generated by the long simple roots.
Let $H$ denote the fundamental group of $\Phi'$, that is,
the quotient of its weight lattice by its root lattice.
We have a short exact sequence
\[
0 \longto i_* (\FM \otimes_\ZM H) \longto \p j_! (\un\FM_{\OC_\mini} [d])
\longto \ic(\ov\OC_\mini,\FM)
\longto 0
\]
Thus
$\p j_! (\un\FM_{\OC_\mini} [d]) \simeq \ic(\ov\OC_\mini,\FM) $
when the characteristic of $\FM$ does not divide $|H|$.
\end{prop}

Thus $\ic(\ov\OC_\mini,\FM)$ is an indecomposable parity complex if the characteristic
of $\FM$ does not belong to the list in Proposition \ref{prop:min} and
does not divide $|H|$.

\def\cprime{$'$} \def\cprime{$'$}

\end{document}